\documentclass[11pt]{article}
\usepackage{color}
\usepackage{amsmath,pdfsync}
\usepackage{amssymb}
\usepackage{amsfonts}
\usepackage{amsthm}
\usepackage{amsmath}
\usepackage{amsthm}
\usepackage{amssymb}
\usepackage{amsfonts}
\usepackage{bm}
\usepackage{titlesec}
\usepackage{verbatim}
\usepackage{subfig}

\usepackage[colorlinks=true,pdfpagemode=none,urlcolor=blue,linkcolor=blue,citecolor=blue]{hyperref}

\usepackage{graphicx}
\usepackage{float}
\usepackage{subfig}
\usepackage{cite}

\usepackage{enumitem}
\usepackage{dsfont}

\usepackage[font=small,labelfont=bf]{caption}
\usepackage[usenames,dvipsnames,svgnames]{xcolor}

\usepackage{tikz,tkz-linknodes}
\usetikzlibrary{arrows,shapes,fit}
\usetikzlibrary{positioning,decorations.pathreplacing,shapes,patterns}
\usetikzlibrary{calc}
\usetikzlibrary{intersections}

\usepackage{pgfplots}

\usepackage{pgfplotstable}
\usepgfplotslibrary{patchplots}

\newcommand{\pdr}[2]{\frac{\partial{#1}}{\partial{#2}}}
\newcommand{\pdrr}[2]{\frac{\partial^2{#1}}{\partial{#2}^2}}
\newcommand{\Rm}{{\mathbb R}}
\newcommand{\eps}{\varepsilon}
\newcommand{\commentout}[1]{}

\numberwithin{equation}{section}

\newtheorem{df}{Definition}[section]
\newtheorem{lemma}[df]{Lemma}
\newtheorem{prop}[df]{Proposition}
\newtheorem{thm}[df]{Theorem}

\numberwithin{theorem}{subsection}

\theoremstyle{definition}

\newcommand{\R}{\mathbb{R}}

\newcommand{\E}{\mathbb{E}}

\newcommand{\farc}{\frac}

\usepackage{url}

\usepackage[letterpaper, top=2.6cm, bottom=2.6cm, left=2.8cm, right=2.4cm, includeheadfoot, head=20pt, foot=20pt]{geometry}

\usepackage[framemethod=TikZ]{mdframed}
\usepackage{lipsum}

\mdfdefinestyle{Solution}{%
    linecolor=gray,
    outerlinewidth=2pt,
    roundcorner=10pt,
    innertopmargin=\baselineskip,
    innerbottommargin=\baselineskip,
    innerrightmargin=20pt,
    innerleftmargin=20pt,
    backgroundcolor=white}

\author{George Papanicolaou\footnote{Email: papanico@stanford.edu} 
\and Lenya Ryzhik\footnote{Email: ryzhik@stanford.edu} \and Katerina Velcheva
\footnote{Email: kati13@stanford.edu. \newline Address for all authors: Department
of Mathematics, Stanford University, Stanford CA 94305, USA}}

\title{Traveling waves in a mean field learning model}
\begin{document}

\maketitle

\begin{abstract} Lucas and Moll have proposed in \cite{LMPaper} 
a system of forward-backward partial differential equations 
that model knowledge diffusion and economic growth. It arises from a microscopic
model of learning for a mean-field type interacting system of individual agents.
In this paper, we prove existence of traveling wave solutions to this system.
They correspond to what is known in economics as balanced growth path solutions. 
We also study the dependence of the solutions and their propagation speed on 
various economic parameters of the system. 
\end{abstract}

\section{Introduction}
\hspace{1cm}

In this paper, we study  an extension of a mean field game model for propagation of
knowledge   and economic growth, proposed by Lucas and Moll \cite{LMPaper}. Mean field models are used in optimal decision making based on stochastic games with a very large population of agents that are statistically identical  \cite{Lionsa, Lionsb,Lehal, Degond1,Degond2, Huang9}. In such  models,  the overall effect of the other agents on a single one can be replaced by an averaged effect, and the optimal  behavior of a single agent can be determined as  a solution 
to an optimal control problem that depends on the distribution of the other agents
and not on their precise cofiguration. The model 
consists of two partial differential equations -- a forward Kolmogorov  equation that keeps track of the distribution of agents, and a backward Hamilton-Jacobi-Bellman (HJB) equation for the value function of the optimal stochastic control problem for each agent. An equilibrium solution is a solution to the coupled system of the two equations valid for all time.

In the model proposed in \cite{LMPaper}, the propagation of knowledge is 
a Markovian process that   involves jumps, but  no diffusion and the economic growth 
is modeled by a production function, that depends on the knowledge of each agent in the 
economy.   The mean field model is
then a coupled system of a forward Kolmogorov  equation for the distribution of knowledge, 
of the linear kinetic type, 
and a backward Hamilton Jacobi Bellman equation for the value function of an individual 
in the economy \cite{LMPaper,MPDE}.  The two equations are coupled through a search function, 
representing the tradeoff between the time spent learning, with no production,
and the time spent producing. We should mention that~\cite{LMPaper} also contains
an excellent survey of the other literature related to models of knowledge diffusion
and growth: without any attempt at completeness we
mention here~\cite{MPDE2,BenPerTon,MR1851053,MR2552289,GDK,MR3554572,MR3268057,Luttmer,MR2904313,PerTon,Staley} 
but an interested reader should consult~\cite{LMPaper} for an illuminating
discussion of various other existing models and further references. 

Balanced Growth Paths (BGP) are special solutions to that system,
valid for all time, corresponding to a constant  growth rate for the economy in equilibrium. In their paper, Lucas and Moll not only propose this interesting model for propagation of knowledge and economic growth, but also study numerically the BGP solutions of the system \cite{LMPaper}. Such solutions are shown to exist in \cite{A1} and \cite{A2} using a fixed point method in 
certain function spaces.

In this paper, we study a model, similar to that  introduced in \cite{LMPaper}, but 
with diffusion added to the process modeling the propagation of knowledge (in a
logarithmic variable), introducing an additional level of uncertainty.  The model is governed by the following system of partial differential equations:
\begin{equation}\label{intrEq1}
\pdr{\psi(t,x)}{t}=\kappa \pdrr{\psi}{x}+\psi(t,x)\int_{-\infty}^x\alpha(s^*(t,y))\psi(t,y)dy-\alpha(s^*(t,x))\psi(t,x)\int_x^\infty \psi(t,y)dy
\end{equation}
and
\begin{equation}\label{intrEq2}
\rho V(t,x)=\pdr{V(t,x)}{t}+\kappa\pdrr{V(t,x)}{x}+\max_{s\in[0,1]}\Big[(1-s)e^x+\alpha(s)\int_x^\infty [V(t,y)-V(t,x)]\psi(t,y)dy\Big].
\end{equation}
Here, $\psi(t,x)$ is the density of the distribution of the
agents, $V(t,x)$ is the value function obtained from the optimal stochastic control problem for each agent, and $s(t,x)$ is the optimal control for each agent. We should 
mention that diffusion was also considered in \cite{A2}  in the original variables, 
where BGP solutions were constructed numerically. We choose to add diffusion after 
the change of variables, because the logarithmic variables are natural from the 
economics point of view.
The economic role of the diffusive term is to incorporate
the change
of knowledge not due to learning via meeting an outside agent but due to innovation and
experimentation, as discussed in~\cite{Luttmer,Staley}. 
Such experimentation may occasionally lead to a small boost in productivity,
and sometimes to a small loss in productivity, and the diffusion term reflects this.

Construction of the BGP solutions in \cite{A1} and \cite{A2}  relied on viewing
(\ref{intrEq1}), in the original variables, as a linear kinetic equation. Here, we
take a different point of view.  
As has already been pointed out in~\cite{LMPaper,MPDE}, in the very special case
when $\alpha(s)$ is independent of $s$, the cumulative distribution function of the agents
is decoupled from (\ref{intrEq2}) and
satisfies the Fisher-KPP (Kolmogorov-Petrovski-Piskunov) equation. This assumption means
that the success of the search does not depend on the fraction of the time spent
searching and is not realistic. However, 
the structure of the full coupled problem without this assumption, for a general search function
$\alpha(s)$, still inherits some
Fisher-KPP features that 
allow us to use a strategy originating in the 
construction of traveling waves for reaction-diffusion equations 
in~\cite{MR807905,MR706099,MR1191008}, albeit with
non-trivial modifications coming from the required estimates for the HJB equation.  

As we explain below, a Balanced Growth Path (BGP) solution of the original system 
corresponds to
a traveling wave solution of (\ref{intrEq1})-(\ref{intrEq2}). These are solutions of the form 
\begin{equation}\label{intrEq3}
\Psi(t,x)=F(x-ct),~~V(t,x)=e^{ct}Q(x-ct), s^*(t,x) = s^*(x-ct),
\end{equation}
where $1-\Psi(t,x)$ is the cumulative distribution function so that $\Psi_x(t,x)=-\psi(t,x)$.
As the BGP solutions, traveling waves are an important class of solutions to the infinite time horizon problem. 
The traveling wave equations satisfied by $F$ and $Q$ in the case $\alpha(s)=\alpha\sqrt{s}$ are written
explicitly in (\ref{may1502})-(\ref{may1504}) below.

The main result of this paper is a proof of existence of traveling waves 
for a specific choice of a search function $\alpha(s)=\alpha\sqrt{s}$.
%
\begin{thm}\label{2.1}
There exist $\rho_0$ that depends on $\kappa$ and $\alpha$, and $\alpha_0$ that depends on $\kappa$,
so that if~$\rho>\rho_0$ and~$\alpha>\alpha_0$, 
then there exists $c$ such that~$0<c<2\sqrt{\kappa\alpha}$ so that 
the system  (\ref{intrEq1})-(\ref{intrEq2}) has 
a solution of the form (\ref{intrEq3}), such that~$F(x)$ is monotonically decreasing, 
$Q(x)$ is monotonically increasing,  and
\begin{equation}\label{intrEq4}
\begin{aligned}
& \lim_{x\rightarrow -\infty}F(x) =1, \lim_{x\rightarrow+ \infty}F(x) = 0,\\
&\lim_{x\rightarrow -\infty} Q(x)~~\hbox{ exists and is positive,}\\
&\lim_{x\rightarrow +\infty} (Q(x)e^{-x}) =\farc{1}{\rho-\kappa}.  
\end{aligned}
\end{equation}
In addition,
 there exist constants $A_{1,2}>0$, $B>0$ such that
\begin{equation}
A_1e^{x}\leq Q(x)\leq A_2e^x + B,
\end{equation}
and $x_0\in\Rm$ so that  $s^*(x)=1$ for all $x<x_0$ and
$s^*(x)<1$ for all $x>x_0$, and $F(x)$ satisfies
\begin{equation}\label{feb320}
 \int_{-\infty}^\infty|F_x|^2 dx<+\infty,
\end{equation}
\end{thm}

The point $x_0$ is the transition point between the agents for $x<x_0$ that do not produce
at all, but rather spend all their time acquiring new knowledge, so that $s^*(x)=1$, 
and agents for~$x>x_0$ that spend  a 
fraction~$s^*(x)\in(0,1)$ of the time acquiring new knowledge and a non-trivial 
fraction~$(1-s^*(x))$ of their time producing. Note that $s^*(x)>0$ for all $x>x_0$. This means
that all agents, no matter how advanced, spend a positive fraction of their time
learning and not just producing. This is a consequence of the assumption 
that $\alpha(s)=\alpha\sqrt{s}$, more specifically, of the fact that~$\alpha'(0)=+\infty$. 
Otherwise, if $\alpha'(0)<+\infty$,
there would be another transition point~$x_1$ so that $s^*(x)=0$ for all
$x>x_1$ -- the very advanced agents would not search at all and will increase their
knowledge only by a random experimentation via diffusion. On the other hand, 
if~$\alpha'(0)=+\infty$ then searching even for a small fraction of the time
gives a "disproportionally large" chance of success, so that even advanced agents
perform a search. This is discussed in more detail in Section~\ref{sec:choice}.

The assumption that $\alpha(s)=\alpha\sqrt{s}$ is convenient to simplify
some considerations but our result can be generalized to concave functions $\alpha(s)$ 
such that $\alpha'(0)=+\infty$ in a straightforward manner. The case of a concave $\alpha(s)$ such
that $\alpha'(0)<+\infty$ can also be studied with a similar approach, except for the 
existence of the second transition point $x_1$ mentioned above. We choose to work
with~$\alpha(s)=\alpha\sqrt{s}$ to keep the presentation as simple as possible while still
interesting from the economics point of view. 
 
The assumption that the discount rate 
$\rho$ is sufficiently large in Theorem~\ref{2.1}  is natural from
the economic intuition. If the discount rate is too small, there is not a sufficient
incentive to produce today, so that the agents would spend all their time just learning 
and we expect that the balanced growth paths do not exist. In particular, we expect that
the transition point $x_0$ moves to $+\infty$ as the discount rate approaches 
the critical value $\rho_0>0$ from above, with the parameters $\alpha$
and $\kappa$ fixed. This is further illustrated numerically  in Section~\ref{sec:num-dis}.

As we have mentioned, when $\alpha(s)=\alpha$ is constant, the equation for $\Psi(t,x)$ 
reduces to the classical Fisher-KPP equation
\begin{equation}\label{feb316}
\Psi_t=\kappa\Psi_{xx}+\alpha\Psi(1-\Psi).
\end{equation}
In that case, traveling waves exist for all $c\ge c_{FKPP}=2\sqrt{\kappa\alpha}$.  
One may wonder if  for the 
full system~(\ref{intrEq1})-(\ref{intrEq2}), there may also exist traveling waves
for all speeds~$c$ larger than some minimal speed $c_*$.  While we only prove here existence
of the wave for a single speed, corresponding to the minimal speed, we can argue that traveling wave does not exist 
for large speeds.   The equation for~$F$, corresponding to the traveling wave profile of $\Psi$ is 
\begin{equation}\label{feb302}
-cF_x-\kappa F_{xx}=\alpha F\int_{-\infty}^x s^*(y)(-F_y)dy,~~F(-\infty)=1,~~F(+\infty)=0.
\end{equation}
As in the classical Fisher-KPP case, as $x\to+\infty$, the solution to (\ref{feb302})
has the asymptotics 
\begin{equation}\label{feb308}
F(x)\sim e^{-\lambda x},\hbox{ as $x\to+\infty$,}
\end{equation}
with the exponential decay rate $\lambda$ related to the propagation speed $c$ by
\begin{equation}\label{feb304}
c\lambda-\kappa\lambda^2=\alpha\gamma,~~
\lambda=\farc{c-\sqrt{c^2-4\alpha\kappa\gamma}}{2\kappa},
\end{equation}
with
\begin{equation}\label{feb306}
\gamma=\int_{-\infty}^\infty s^*(y)(-F_y)dy\le \int_{-\infty}^\infty (-F_y)dy= 1.
\end{equation}
The difference with the standard Fisher-KPP situation is that $\gamma$ is not explicit
but the decay rate~$\lambda$ and the traveling wave speed $c$ are still 
related by (\ref{feb304}).
It follows, in particular, that~$\lambda<1$ if~$c>\alpha+\kappa$. However, the value
function $Q(x)$ has the asymptotics $Q(x)\sim e^x$ as $x\to+\infty$, that comes both
from (\ref{intrEq2}) and its economic interpretation. 
In addition, for the traveling wave
solutions to make sense, the expected benefit of the search, given by the integral
\begin{equation}\label{feb404}
\int_x^\infty[Q(y)-Q(x)](-F_y(y))dy=\int_x^\infty Q'(y)F(y)dy,
\end{equation}
that appears in the right side of (\ref{intrEq2}),
must be finite. This in incompatible with (\ref{feb308}) if $\lambda<1$. It follows that 
traveling
waves with speeds~$c>\kappa+\alpha$ can not exist. We expect 
that there exists an interval of speeds $[c_{min},c_{max})$ so that 
(\ref{intrEq1})-(\ref{intrEq2}) has traveling wave solutions for 
all~$c\in[c_{min},c_{max})$. This gives a limit on how fast economy may grow
along a balanced growth path, within the parameters of this model.  

The upper bound $c<c_{FKPP}=2\sqrt{\kappa\alpha}$ in Theorem~\ref{2.1}
is an immediate consequence of the following relation between 
the speed $c$ and the wave profile constructed in Theorem~\ref{2.1}, that was
conjectured in~\cite{MPDE}, as  
a direct analog of the minimal front speed formula $c_{FKPP}=2\sqrt{\kappa\alpha}$ for (\ref{feb316}).
\begin{prop}\label{prop-speed-2.2} 
The speed $c$, the search function $s^*(x)$ and $F(x)$ constructed
in Theorem~\ref{2.1} are related by
\begin{equation}\label{speedRelation}
c=2\sqrt{\kappa\alpha\gamma}, 
\end{equation}
with $\gamma<1$ as in (\ref{feb306}).  
\end{prop}
 


The assumption of Theorem ~\ref{2.1}, that the search effectiveness 
parameter $\alpha$ is large is also natural from the economic intuition.
The proof of Proposition~\ref{prop-speed-2.2} shows not only that
(\ref{speedRelation}) holds but also 
that the 
traveling wave $F(x)$ constructed in Theorem~\ref{2.1} satisfies (\ref{feb308})
with
\begin{equation}\label{feb402}
\lambda=\farc{c}{2\kappa}=\sqrt{\farc{\alpha\gamma}{\kappa}}<\sqrt{\farc{\alpha}{\kappa}},
\end{equation}
in agreement with (\ref{feb304}), due to (\ref{speedRelation}). 
Therefore, if $\alpha\le \kappa$ is too small, then $\lambda\le 1$ and
the integral in (\ref{feb404}) would, once again, blow up. As we show below,
the transition point $x_0$ and the integral in (\ref{feb404}) are related
by
\begin{equation}\label{feb416}
e^{x_0}=\int_{x_0}^\infty Q'(y)F(y)dy.
\end{equation}
As $\alpha$ approaches a critical 
value $\alpha_0$ from above, we expect that
the integral in (\ref{feb404}) blows up for any $x$ fixed. It follows then from
(\ref{feb416}) that we must have $x_0\to+\infty$.
In this sense, the effect of small alpha is 
similar to that of a small discount rate $\rho$: all agents search
rather than produce, though for a different reason.
Now, as $\alpha$ approaches $\alpha_0$ from above, 
the chance of a successful search is small (even though it does not vanish
as $\alpha\downarrow\alpha_0$), and more and more skilled
agents have to search, to keep the economy growing along a balanced path.  
This is also illustrated numerically in Section~\ref{sec:num-dis}.

Another consequence of (\ref{feb402}) and the requirement that $\lambda>1$
is that the traveling waves constructed in 
Theorem~\ref{2.1} satisfy
\begin{equation}\label{feb410}
c\ge 2\kappa.
\end{equation}
This condition holds also for any traveling wave, not just those we construct in
that theorem. Indeed, any traveling wave that moves with a speed $c$ satisfies
the decay estimate in (\ref{feb308}) with~$\lambda$ related to 
$c$ via (\ref{feb304})-(\ref{feb306}). The requirement that $\lambda>1$ then implies
the lower bound on the speed in (\ref{feb410}). 

The limitation in Theorem~\ref{2.1} that $\rho$ is sufficiently large is also a limit on how large
the diffusivity $\kappa$ can be for a given value of the discount rate $\rho$.
One can already see that from the behavior of $Q(x)$ as~$x\to+\infty$ in (\ref{intrEq4}).
Mathematically, this comes from the requirement that $\rho$ is larger than the principal
eigenvalue of a certain linear operator with a diffusion term $\kappa\partial_{x}^2$.
A toy model for this phenomenon is that the solution to 
\begin{equation}\label{feb312}
\pdr{\phi}{t}+\rho\phi=\kappa\Delta\phi+e^x
\end{equation}
with, say, zero initial condition, is given by
\[
\phi(t,x)=\farc{1}{\kappa}e^x(e^{(\kappa-\rho) t}-e^{-\rho t}).
\]
Therefore, 
for a balanced growth path to exist, the discount rate $\rho$ has to be larger than $\kappa$
-- otherwise, $\phi(t,x)$ blows up as $t\to+\infty$. 
From the economics point 
of view, this means that, since the knowledge gained
by diffusion already leads to an exponential growth in time, 
the discount rates need to be sufficiently high for a balanced
growth path to exist,
so the total production would not blow up.

As we have mentioned, when $\kappa=0$ the balanced growth paths
in the original non-logarithmic variables, which correspond exactly to the
traveling wave solutions for (\ref{intrEq1})-(\ref{intrEq2}),
have been constructed in \cite{A1,A2} using completely different techniques. 
These are, however, slightly different objects from the traveling wave
we construct in Theorem~\ref{2.1} for $\kappa>0$, as the case $\kappa=0$
is special even for 
the Fisher-KPP equation (\ref{feb316}), in the following sense. Generally, for $\kappa>0$,
traveling waves for (\ref{feb316}) exist for all speeds $c\ge c_*=2\sqrt{\kappa\alpha}$.
The minimal speed is special in that solutions to the Cauchy problem for (\ref{feb316})
with all sufficiently 
rapidly decaying initial conditions converge to a translate of the wave moving
with the minimal speed, while traveling waves for $c>c_*$ represent the long time
behavior of the solutions to (\ref{feb316}) that have exactly the same exponential
decay at~$t=0$ as the corresponding traveling wave. The economic
interpretation of the former case is that the initial distribution of
the logarithm of knowledge $\psi(0,x)$ has a right tail that decays quickly
and may, in fact, have bounded support. The interpretation of the latter
case is that this initial distribution not only has an unbounded support but, in
fact, has a precise exponential right tail meaning that the initial
distribution of the level of knowledge $e^x$ has a fat right tail (it
follows a power law). The role of positive diffusion in overcoming the need for
heavy tails of the initial distributions is discussed in detail in~\cite{Luttmer}.  
On the other hand, when $\kappa=0$,
so that $c_*=0$, there is no traveling wave solution for (\ref{feb316})
moving with the minimal speed but traveling waves do exist for all $c>0$. The
balanced growth paths constructed in~\cite{A1,A2} for $\kappa=0$ are
the analogs of these "super-critical" Fisher-KPP waves for $\kappa=0$. 
On the other hand,
the traveling wave constructed in Theorem~\ref{2.1} is the analog of the minimal speed
wave for the Fisher-KPP equation and thus does not exist for~$\kappa=0$ but only
for $\kappa>0$. In that sense,
Theorem~\ref{2.1} is a complementary result to \cite{A1,A2}.

The methods of the present paper also allow to study the long time existence of solutions
to the time-dependent coupled forward-backward problem (\ref{intrEq1})-(\ref{intrEq2}).
This will be discussed elsewhere~\cite{PRV2}.

{\bf Organization of the paper.}
In Section \ref{model}, we review the mean field learning model, presented in \cite{LMPaper}, 
and formulate the mean filed system with diffusion added after the logarithmic  
change of variable. We also discuss the formulation for the 
specific choice of a search function $\alpha(s)=\alpha\sqrt{s}$.  
 
In Section \ref{FIP}, we prove Theorem~\ref{2.1}. As we have mentioned,
the proof uses a general 
strategy for the construction of traveling waves
originating in~\cite{MR807905,MR706099,MR1191008} and is
in two steps: first, we  consider a suitable approximate problem on 
a finite interval $[-a,a]$, for a sufficiently large $a$. 
The key step is to show that a solution $(F^a,Q^a,c^a)$ to the approximate 
problem exists. This is done by obtaining 
a priori bounds on the solutions and a degree argument. The a priori bounds for the
coupled system is the main nontrivial difficulty in the present problem
compared to the standard reaction-diffusion scalar equations.
Next, using the a priori bounds on the solutions to the approximate problem on 
finite intervals, we pass to the limit along a subsequence $a_n\to+\infty$ and 
show that~$(F^{a_n},Q^{a_n},c^{a_n})$ converge uniformly on compact sets 
to a solution $(F,Q,c)$ to the traveling wave system (\ref{may1502}), and that the boundary 
conditions~(\ref{may1504}) are also satisfied by the functions $F$ and $Q$.  

In Section \ref{Numerics}, we describe an iterative finite difference numerical
algorithm solving the problem on a finite interval $[-a,a]$ and 
discuss the properties of the numerical solutions. The 
simulations show clearly the validity of the result and clarify the dependence of the solutions 
on various parameters that enter the problem. 
We should mention that different iterative numerical algorithms for BGP solutions of the mean 
field model are given 
in \cite{A1}, \cite{A2} and \cite{LMPaper}. 
One difference with the present paper is in the procedure that finds  numerically
the wave speed.

{\bf Acknowledgment.} LR was supported by the NSF grants DMS-1613603 and DMS-1910023.
We are indebted to Henri Berestycki and Benjamin Moll for illuminating discussions.

\section{The mean field learning model}\label{model}

In this section, we recall the basics of the  mean field learning model introduced
in \cite{LMPaper}. In addition, we reformulate the model in the logarithmic
variables and add diffusion in the knowledge space. We also define the notion of
traveling wave solutions correspond to  
the balanced growth paths in the original variables.
%
%
\subsection{The non-diffusive model}\label{descriptionOfTheModel}

Consider a 
population of agents, such that each agent has a certain knowledge $z\geq 0$ at a given
time $t\ge 0$.
An agent can either produce or learn at each moment of time, and we denote
by~$s(t,z)\in[0,1]$ the fraction of time an agent with 
knowledge~$z\geq 0$ spends learning
on a time interval $[t,t+\Delta t]$, so
that $(1-s(t,z))$ is the fraction of time he spends producing on this time interval.
His total production between the times~$t$ and~$t+\Delta t$ is then
\begin{equation}
[1-s(t,z)]z\Delta t.
\end{equation}
Agents in the economy learn by meeting other agents, with a higher production 
knowledge. In order to
describe the meetings, let $\Phi(t,z)$ be the fraction of the agents with knowledge less 
or equal to $z$ at time~$t$,
and let $\phi(t,z)=\Phi_z(t,z)$ be the corresponding density. 
The probability that the search by an agent~$A$  
with knowledge $z\geq 0$ is successful 
on a time interval $(t,t+\Delta t)$ is  $\alpha(s(t,z))\Delta t$, 
where~$\alpha(s):[0,1]\rightarrow \R^+$ is a given concave function. Given that the search
is successful, the probability that $A$ encounters an agent $B$ with knowledge in the interval
$(z',z'+\Delta z')$
is proportional to $\phi(t,z')\Delta z'$ -- this is the mean field nature of the model.
If the production knowledge of agent $A$ is lower than the production knowledge of agent 
$B$, then 
agent $A$ updates his production knowledge to that of agent $B$.  The overall 
balance leads to the  
following nonlinear kinetic equation for the density $\phi(t,z)$:
\begin{align}\label{FNoDif}
\frac{\partial \phi(t,z)}{\partial t} 
=& -\alpha(s(t,z))\phi(t,z)\int_z^\infty \phi(t,y)dy + 
\phi(t,z)\int_0^z\alpha(s(t,y))\phi(t,y)dy.
\end{align}

An agent with knowledge $z$ at time $t$ chooses the search time 
$s(t,z)$, so as to maximize the expected total production (value function) 
$V(t,z)$, discounted in time:
\begin{equation}\label{may1308}
V(t,z)=\max_{s\in\mathcal{A}}\E\Big{\{}\int_t^T e^{-\rho(\tau-t)}z(\tau)[1-s(\tau,z(\tau))]+V_T(z)\big| z(t)=z \Big{\}}.
\end{equation}
Here $\rho>0$ is a discount parameter and $\mathcal{A}$ is the set of admissible control 
functions, $T>0$ is a given terminal time, and $V_T(z)$ is a prescribed 
terminal value.   
The value function $V(t,z)$ satisfies a Hamilton-Jacobi-Bellman equation:
\begin{equation}\label{VNoDif}
\rho V(t,z) = \pdr{V(t,z)}{t} + 
\sup_{s\in[0,1]}\bigg\{ (1-s)z +\alpha(s)\int_z^\infty [V(t,y) - V(t,z)] 
\phi (t,y) dy 
\bigg\},
\end{equation}
supplemented by the terminal condition $V(T,z)=V_T(z)$
We will denote by $s^*(t,x)$ the optimal control in (\ref{VNoDif}).
An informal derivation of (\ref{FNoDif}) and (\ref{VNoDif})
is given in~\cite{LMPaper}.



In the economics context, especially since we are soon going to introduce the
diffusion of knowledge, it is natural to 
consider an exponential change of variables $\phi(t,z)=\psi(t,\log z)/z$, so that
the function~$\psi(t,x)$ is also a density but in the logarithmic variables
\begin{equation}
1=\int_0^\infty \phi(t,z)dz=\int_0^\infty \psi(t,\log z)\farc{dz}{z}=\int_{-\infty}^\infty \psi(t,x)dx.
\end{equation} 
This change of variables  transforms
(\ref{FNoDif}), (\ref{VNoDif}) 
into the following system:
\begin{equation}\label{expF}
\pdr{\psi(t,x)}{t}=\psi(t,x)\int_{-\infty}^x\alpha(s^*(t,y))\psi(t,y)dy-\alpha(s^*(t,x))\psi(t,x)\int_x^\infty \psi(t,y)dy.
\end{equation}
\begin{equation}\label{expV}
\rho V(t,x)=\pdr{V(t,x)}{t}+\max_{s\in[0,1]}\Big[(1-s)e^x+\alpha(s)\int_x^\infty [V(t,y)-V(t,x)]\psi(t,y)dy\Big],
\end{equation}
with the initial condition $\psi(0,x)=\phi(0,e^x)e^x$,
and the terminal condition~$V(T,x)=V_T(e^x)$. 

\subsection{The diffusive model}

Equations (\ref{expF})-(\ref{expV}) assume that the only changes in the productivity of the
agents come from their interactions. It is reasonable from the economics
point of view to assume that even in the absence of such interactions the productivity 
of each agent undergoes
some diffusion, so that the agents learn not only from each other but 
also through experimenting, and it is natural to do that in the logarithmic variables,
as in (\ref{expF})-(\ref{expV}).  
Adding diffusion to both equations transforms the system to
 \begin{equation}\label{FDif}
\pdr{\psi(t,x)}{t}=\kappa \pdrr{\psi}{x}+\psi(t,x)\int_{-\infty}^x\alpha(s^*(t,y))\psi(t,y)dy-\alpha(s^*(t,x))\psi(t,x)\int_x^\infty \psi(t,y)dy
\end{equation}
and
\begin{equation}\label{VDif}
\rho V(t,x)=\pdr{V(t,x)}{t}+\kappa\pdrr{V(t,x)}{x}+\max_{s\in[0,1]}\Big[(1-s)e^x+\alpha(s)\int_x^\infty [V(t,y)-V(t,x)]\psi(t,y)dy\Big].
\end{equation}
This is the system (\ref{intrEq1})-(\ref{intrEq2}). 
%

We will also make use of the cumulative distribution function 
\begin{equation}\label{BCPsi2}
\Psi(t,x)=\int_x^\infty \psi(t,y)dy,~~\Psi(-\infty)=1,~\Psi(+\infty)=0.
\end{equation}
A straightforward computation shows that $\Psi(t,x)$ 
satisfies the following integro-differential equation:
\begin{equation}\label{Psi2}
\begin{aligned}
\pdr{\Psi}{t}-\kappa\pdrr{\Psi}{x}&= -\Psi(t,x)\int_{-\infty}^x\alpha(s^*(t,y))\Psi_y(t,y)dy.
\end{aligned}
\end{equation}
The equation for the value function $V(t,x)$ in terms of $\Psi(t,x)$ is
\begin{equation}\label{V2}
\rho V(t,x)=\pdr{V(t,x)}{t}+\kappa\pdrr{V(t,x)}{x}+\max_{s\in[0,1]}\Big[(1-s)e^x+\alpha(s)\int_x^\infty [V(t,y)-V(t,x)](-\Psi_y(t,y))dy\Big].
\end{equation}

Equation (\ref{Psi2}) should be supplemented by an initial condition for $\Psi(0,x)$
and (\ref{V2}) should come with a terminal condition for $V(T,x)=V_T(x)$ 
at some $T>0$. Existence
of the solutions of the resulting forward-backward in time problem will be discussed 
elsewhere~\cite{PRV2}. One natural terminal condition is $V_T(x)=0$, 
as there is no time left to produce at the end. 
This, however, is not the only possibility as one could also try to choose $V_T(x)$
so as to approximate
the solution to the infinite time horizon problem with $T=+\infty$, so that 
$V(t,x)$ in (\ref{may1308}) is re-defined as
\begin{equation}\label{may1310}
V(t,z)=\max_{s\in\mathcal{A}}\E\Big{\{}\int_t^\infty 
e^{-\rho(\tau-t)}z(\tau)[1-(\tau,z(\tau))]\big| z(t)=z \Big{\}}.
\end{equation}
A very interesting question, to be addressed in~\cite{PRV2}, is if the  
pair of solutions~$\Psi_T(t,x)$,~$V_T(t,x)$ defined on the time interval $0\le t\le T$, with 
some prescribed terminal
conditions, have a well-defined limit~$\Psi(t,x)$,~$V(t,x)$ as $T\to+\infty$. This would
be a natural candidate for a "correct" solution to the infinite horizon problem,
without an explicit terminal condition for~$V(t,x)$.  

As we have mentioned, in the special case when $\alpha(s)=\alpha$ is a constant, the system
(\ref{Psi2})-(\ref{V2}) decouples, and~(\ref{Psi2}) becomes the 
classical Fisher-KPP equation (\ref{feb316}).
Its solutions in the long time limit converge to traveling waves moving with
the speed $c_*=2\sqrt{\kappa\alpha}$. This direct analogy to the Fisher-KPP type
problems works only in the special case when $\alpha(s)$ is constant. However,
in general, one still expects that, as in the FKPP case,
the long time behavior of the solutions to~(\ref{Psi2}) is governed to the
leading order by the linearization as~$x\to+\infty$:
\begin{equation}\label{june504}
\begin{aligned}
\pdr{\tilde\Psi}{t}-\kappa\pdrr{\tilde\Psi}{x}&= R(t)\tilde\Psi(t,x),~~
R(t)=\int_{-\infty}^\infty\alpha(s^*(t,y))\Psi_y(t,y)dy,
\end{aligned}
\end{equation}
Note that, unlike in the true FKPP case, the linearized 
equation (\ref{june504}) is not closed in general as the
rate $R(t)$ depends on the function $V(t,y)$ as well. Nevertheless, 
it is natural to conjecture that 
solutions to the full problem still belong to the so called class of pulled
fronts \cite{EVS}, and significant intuition can be gained from the Fisher-KPP
analogy.


\subsection{The choice of the search function}\label{sec:choice}

The maximization problem in (\ref{V2})
is of the form
\begin{equation}\label{V2bis}
\max_{s\in[0,1]}\Big[(1-s)+B\alpha(s)\Big],
\end{equation}
with
\begin{equation}\label{may1318}
B=e^{-x}\int_x^\infty [V(t,y)-V(t,x)](-\Psi_y(t,y))dy,
\end{equation}
so that the optimal $s$ is given by
\begin{equation}\label{sVal}
s^* =s^*(B)=\begin{cases}
0,~~B\leq \dfrac{1}{\alpha'(0)},\\
\beta(\dfrac{1}{B}),~~\dfrac{1}{\alpha'(0)}<B\leq\dfrac{1}{\alpha'(1)},\\
1,~~B> \dfrac{1}{\alpha'(1)},
\end{cases}
\end{equation}
where $\beta = (\alpha ')^{-1}$. In order to avoid the situation where
agents of sufficiently advanced knowledge do not search at all, it is natural to
assume that
$\alpha'(0)=+\infty$. To simplify some computations, we will make an assumption
that $\alpha(s)=\alpha\sqrt{s}$ with some $\alpha>0$. Generalizations of our results to
a general concave function $\alpha(s):[0,1]\to[0,1]$ with $\alpha'(0)=+\infty$ 
are quite straightforward.
%
%
Now, equations (\ref{V2}) and (\ref{Psi2}) become
\begin{equation}\label{Veq}
\rho V(t,x)=\pdr{V(t,x)}{t}+\kappa\pdrr{V(t,x)}{x}+e^x\max_{s\in[0,1]}\Big[(1-s^2)+\alpha s e^{-x}\int_x^\infty [V(t,y)-V(t,x)](-\Psi_y(t,y))dy\Big],
\end{equation}
and 
\begin{equation}\label{Psieq}
\pdr{\Psi}{t}-\kappa\pdrr{\Psi}{x}= -\alpha\Psi(t,x)\int_{-\infty}^xs^*(t,y)\Psi_y(t,y)dy,
\end{equation}
To simplify (\ref{Veq}), we introduce the auxiliary functions
\begin{eqnarray}\label{rdef}
&&r(t,x)=\farc{\alpha}{2} e^{-x}  \int_x^\infty [V(t,y)-V(t,x)](-\Psi_y(t,y))dy,
\\
\label{Hdef}
&&H(r(t,x))=\max_{s\in[0,1]}\Big[(1-s^2) +\alpha s e^{-x}\int_x^\infty [V(t,y)-V(t,x)]\psi(t,y)dy\Big]\\
&&~~~~~~~~~~~~~=\max_{s\in[0,1]}\Big[(1-s^2) +2 s r(t,x)\Big],\nonumber
\end{eqnarray}
so that $H(r)$ and the maximizer $S^*(r)$ are given by
\begin{equation}\label{Hdef2}
H(r)=\left\{\begin{matrix} 2r,~r>1,\cr 1+r^2,~0<r<1,\cr 1,~r<0.\cr\end{matrix}\right.
~~~~~~~
S^*(r)=\left\{\begin{matrix} 1,~r>1,\cr r,~0<r<1,\cr 0,~r<0.\cr\end{matrix}\right.
\end{equation}
Now, we can write (\ref{Veq}) - (\ref{Psieq}) as
\begin{equation}\label{eqFnw}
\pdr{\Psi}{t}-\kappa\pdrr{\Psi}{x}= \alpha\Psi(t,x)\int_{-\infty}^xS^*(r(t,y))(-\Psi_y(t,y))dy,
\end{equation}
and
\begin{equation}\label{eqVnw}
\rho V(t,x)=\pdr{V(t,x)}{t}+\kappa\pdrr{V(t,x)}{x}+e^x H(r(t,x)).
\end{equation}
Thus, the new formulation of the problem are equations (\ref{eqFnw})-(\ref{eqVnw})
for $\Psi(t,x)$ and $V(t,x)$, with the function $r(t,x)$ defined by (\ref{rdef}),
and $H(r)$ and $S^*(r)$ given by (\ref{Hdef2}).


\subsection{The traveling wave solutions}

The infinite time horizon problem has special solutions that in the original
variables are known as the balanced growths path (BGP).
These are solutions to (\ref{FNoDif}), (\ref{VNoDif}) 
of the form 
\begin{equation}
\phi(t,z)=e^{-\gamma t}f(ze^{-\gamma t}),~~V(t,z)=e^{\gamma t}v(ze^{-\gamma t}),
~~s(t,z)=\sigma(ze^{-\gamma t}),
\end{equation}
with some $\gamma>0$, and $f(x),v(x)\in {C}^1(\Rm)$ and $\sigma(x)\in{C}(\Rm)$.  
The BGP solutions are interesting from the economics point of view since they 
give a constant growth rate for the economy, but they also give a well-defined 
solution to the infinite time horizon problem, and it is natural to conjecture,
from the numerical evidence, that they should be the long time
limit of the finite horizon problems on a time interval $[0,T]$ as $T\to+\infty$,
with a proper terminal condition $V_T(x)$. This is similar to the stability of the
Fisher-KPP traveling waves. 
 
It has been shown in \cite{A1}
that there exists $\mu_0>0$ so that the~BGP solutions with the asymptotics
\begin{equation}\label{may1302}
\phi(z)\sim z^{-\mu},\hbox{ as $z\to+\infty$},
\end{equation}
exist for all $0<\mu<\mu_0$, with a corresponding growth rate $\gamma(\mu)\in(0,\rho)$.
After the exponential change of variables, 
a BGP solution defined for $z\ge 0$ transforms to a traveling wave solution 
for~$x\in\R$ that moves with a constant speed equal to the growth rate $\gamma$:
\begin{equation}\label{may1304}
\psi(t,x)=e^x\phi(t,e^x)=e^{x-\gamma t}f(e^{x-\gamma t})=\Psi(x-\gamma t).
\end{equation}  
Traveling waves are solutions to the system (\ref{eqFnw})-(\ref{eqVnw}) of the form
\begin{equation}
\Psi(t,x)=F(x-ct),~~V(t,x)=e^{ct} Q(x-ct),~~r(t,x)=R(x-ct).
\end{equation}
They correspond to the balanced growth paths before the logarithmic change
of variables.  Note that if $F(x)$, $Q(x)$ and $R(x)$ form a traveling wave, 
with the corresponding search function $s^*(x)$, then for any fixed 
shift $y\in\Rm$, the functions 
\begin{equation}\label{may1320}
F_y(x):=F(x-y),~s_y^*(x):=s^*(x-y),~Q_y(x):=e^yQ(x-y),~R_y(x):=R(x-y)
\end{equation}
also form a traveling wave solution, so that traveling waves form
a one parameter family, which is a typical situation in the theory of traveling waves.
The only difference is that the value function~$Q(y)$ is transformed slightly different
under a shift by $y$. 


%
%
%
A traveling wave    
satisfies the following system:
\begin{equation}\label{may1502}
\begin{aligned}
&-cF_x-\kappa F_{xx}=\alpha F(x)\int_{-\infty}^{x}s^*(y)(-F_y(y))dy,\\
&\rho  Q(x)=c Q(x)-c \frac{\partial Q(x)}{\partial x}+ 
\kappa\frac{\partial^2 Q(x)}{\partial x^2}+e^{x} H(R(x)),\\
&R(x)=\frac{\alpha}{2} e^{-x}  \int_{x}^\infty [Q(y)-Q(x)](-F_y(y))dy.
\end{aligned}
\end{equation}
with $s^*(x)=\min[1,R(x)]$, and with boundary conditions 
\begin{equation}\label{may1504}
\begin{aligned}
& \lim_{x\rightarrow -\infty}F(x) =1, \lim_{x\rightarrow \infty}F(x) = 0,\\
&\lim_{x\rightarrow -\infty} Q(x)~~\hbox{ exists and is positive,}\\
&\lim_{x\rightarrow +\infty} (Q(x)e^{-x}) =\farc{1}{\rho-\kappa}.  
\end{aligned}
\end{equation}
Theorem~\ref{2.1} is the existence result for this system. 
%

The proof of Theorem \ref{2.1} is presented in Section~\ref{FIP}.
As we have mentioned in the introduction, it proceeds in two steps: 
first, we consider an approximate problem on a finite interval $[-a,a]$, 
with~$a\gg 1$, and an additional normalization~$F^a(0)=1/2$ needed to 
fix the speed $c^a$. We obtain
a priori bounds on the solutions and use a degree argument to show that there exists
a solution~$(F^a,Q^a,c^a)$ to the approximate problem. In the second step, using the a 
priori bounds on the finite intervals, we pass to the limit along a subsequence 
$a_n\to+\infty$ and show that 
$(F^{a_n},Q^{a_n},c^{a_n})$ converge uniformly on compact sets to a solution
$(F,Q,c)$ to (\ref{may1502}), and that the boundary conditions (\ref{may1504}) are also
satisfied by the functions $F$ and $Q$.  Proposition~\ref{prop-speed-2.2} is proved at the end of Section~\ref{FIP}.

\section{Existence of a traveling wave solution}\label{FIP}
In this section, we prove Theorem~\ref{2.1}. 


\subsection{The finite interval problem}\label{sec:finite-interval}

In the first step, 
we restrict the system (\ref{may1502}) to a finite interval $[-a,a]$, 
with $a>0$ and consider the following approximate problem for
the functions $F^a(x)$, $Q^a(x)$ and $R^a(x)$,
and a speed $c^a$: 
\begin{eqnarray}\label{Feq}
&&-c^aF_x^a-\kappa F_{xx}^a=\alpha F^a(x)\int_{-a}^{x}s_a^*(y)(-F_y^a(y))dy,
\\
\label{Qeq}
&&\rho  Q^a(x)=c^a Q^a(x)-c^a \frac{\partial Q^a(x)}{\partial x}+ 
\kappa\frac{\partial^2 Q^a(x)}{\partial x^2}+e^{x} H(R^a(x)),
\\
&&\label{may1510}
R_a(x)=\frac{\alpha}{2} e^{-x}  \int_{x}^a [Q^a(y)-Q^a(x)](-F^a_y(y))dy,
\end{eqnarray}
with $s_a^*(x)=\min[1,R^a(x)]$ and with the boundary conditions 
\begin{eqnarray}\label{BCF}
&& F^a(-a) =1,~~F^a(a) = 0,
\\
\label{BCQ}
&&Q_x^a(-a)=0,~~Q_x^a(a)=Q^a(a).
\end{eqnarray}
In addition, we impose a normalization for $F^a$:
\begin{equation}\label{may1512}
F^a(0)=1/2,
\end{equation}
that is needed to obtain uniform bounds on the speed $c^a$ that at the moment is assumed
to be unknown.  
Let us define $x^a_0$ as  
\begin{equation}\label{may2226}
x^a_0=\sup\{x:~s_a^*(x)=1\}.
\end{equation}
The main result of this step is the following proposition:
\begin{prop}\label{FiniteInterval}
There exists $a_0>0$ so that for all $a>a_0$ there exists a constant $c_a\in\R$ for which the 
system (\ref{Feq})-(\ref{may1510}) has a solution such that $F^a(x)$ and $R^a(x)$ are 
monotonically decreasing,~$Q^a(x)$ is increasing, and the boundary conditions 
(\ref{BCF})-(\ref{BCQ}), as well as the normalization~(\ref{may1512}), hold.
Moreover, there exists a  constant  $C$ independent of $a$, and $a_0>0$ such that for 
all~$a>a_0$ we have  
\begin{equation}\label{may1514}
|c^a|+ \int_{-a}^a|F_x^a|^2 dx\leq C.
\end{equation}
There also exist constants $A_1$, $A_2$, $B$, $x_0^-$ and $x_0^+$ that do not depend on
$a$, such that for all~$a>a_0$ we have
\begin{equation}
A_1e^{x}\leq Q(x)\leq A_2e^x + B,
\end{equation}
and
\begin{equation}
x_0^-<x_0^a<x_0^+.
\end{equation}
\end{prop} 
The proof of this proposition relies on a Leray-Schauder degree argument:
we consider a family of systems of 
equations 
\begin{eqnarray}\label{tausis}
&&-c_a^\tau \pdr{F_a^\tau}x=\kappa \pdrr{F_a^\tau}{x}+\alpha F_a^\tau(x)
\int_{-a}^x[(1-\tau )+\tau s^*_{a,\tau}(y)]\big(-\pdr{F_a^\tau(y)}{y}\big)dy\\
\label{may1708}
&&(\rho -c_a^\tau)Q_a^\tau+c_a^\tau \pdr{Q_a^\tau}{x}-\kappa \pdrr{Q_a^\tau}{x}=\tau e^x H(R_a^\tau(x)),
\end{eqnarray}
with
\begin{equation}\label{may1702}
R_a^\tau(x)=(1-\tau)+\tau\frac{\alpha}{2}e^{-x}\int_x^a[Q_a^\tau(y)-Q_a^\tau(x)]
\big(-\pdr{F_a^\tau(y)}{y}\big)dy,
~~s_{a,\tau}^*(x)=\min(1,R_a^\tau(x)),
\end{equation}
and with the boundary conditions
\begin{eqnarray}\label{may1704}
&&F_a^\tau(-a)=1, ~F_a^\tau(a)=0, 
\\
\label{may1710}
&& \pdr{Q_a^\tau(-a)}{x}=0,~~\pdr{Q_a^\tau(a)}{x}=Q_a^\tau(a),
\end{eqnarray}
together with the normalization 
\begin{equation}\label{may1706}
F_a^\tau(0)=\frac{1}{2}.
\end{equation}
This family is 
parametrized by $\tau\in [0,1]$, so that at $\tau=0$ it reduces to the classical
Fisher-KPP equation
\[
-c_a^0\pdr{F_a^0}x=\kappa \pdrr{F_a^0}{x}+F_a^0(1-F_a^0),
\]
and $Q_a^0(x)=0$, $R_a^0(x)=s_{a,0}^*(x)=1$ for all $x\in[-a ,a]$, while 
at $\tau=1$ the system (\ref{tausis})-(\ref{may1706}) is exactly the problem
(\ref{Feq})-(\ref{may1512}) that
we are interested in.
We will show that the above system has a solution for all $\tau\in[0,1]$, and,
in particular, for $\tau=1$. 
The main difficulty in the proof of Proposition~\ref{FiniteInterval} is to obtain 
the uniform a priori bounds on the solutions to (\ref{tausis})-(\ref{may1706}) that do not
depend on $a$.

 \subsubsection{A priori bounds on a finite interval}\label{APBounds}
 
We now prove the required a priori bounds for the solutions to (\ref{tausis})-(\ref{may1706})
that are uniform in the parameter $\tau$ and do not depend on $a$ for $a>a_0$. 
 
\subsubsection*{The monotonicity of $F_a^\tau$}
 
We start by establishing monotonicity of $F_a^\tau$  for all $\tau\in[0,1]$.
\begin{lemma}\label{Fmonotone}
The function $F_a^\tau(x)$, satisfying (\ref{tausis}) together with the boundary
conditions~(\ref{may1704}) and normalization (\ref{may1706})
is positive on $(-a,a)$ and decreasing in $x$ for all $\tau\in[0,1]$.
\end{lemma}
\begin{proof}
It is helpful to write
\begin{equation}\label{taueq}
\begin{aligned}
-c_a^\tau \pdr{F_a^\tau}x&-\kappa \pdrr{F_a^\tau}{x} = \alpha F_a^\tau(x)\int_{-a}^x 
[(1-\tau) + \tau s_{a,\tau}^*(y)](-\pdr{F_a^\tau(y)}{y})dy \\ 
&= \alpha (1-\tau) F_a^\tau(x)(1-F_a^\tau(x)) + 
\alpha\tau F_a^\tau(x)\int_{-a}^x    s^*_{a,\tau}(y)(-\pdr{F_a^\tau(y)}{y})dy.
\end{aligned}
\end{equation}
Note that for $\tau = 0$ this is the Fisher-KPP equation
\begin{equation}
-c_a^0\pdr{F_a^0}x-\kappa \pdrr{F_a^0}{x} = \alpha F_a^0(x)\int_{-a}^x(-\pdr{F_a^0(y)}{y})dy 
= \alpha F_a^0(x)(1-F_a^0(x)),
\end{equation}
with the boundary conditions (\ref{may1704}),
for which we know that the solution $F_a^0(x)$ is positive on~$(-a,a)$
and is strictly decreasing in~$x$,
so that $\partial_xF_a^0(x)<0$ for
all $x\in[-a,a]$. By continuity, we have $0<F_a^\tau(x)<1$ for all $x\in(-a,a)$ and
$\partial_xF_a^\tau(x)<0$ for all~$x\in[-a,a]$ for $\tau>0$ sufficiently small. 
Furthermore, note that if $x_0$ is the local minimum or maximum of
$F_a^\tau(x)$ that is closest to~$(-a)$, then $\partial_xF_a^\tau(x)$ does not change sign
on $(-a,x_0)$, so that the integral term in the right side of~(\ref{taueq}) is either strictly
positive if $x_0$ is a minimum, or strictly negative if $x_0$ is a maximum, which immediately
gives a contradiction unless $F_a^\tau(x_0)<0$. To rule out this possibility, 
let~$\tau_1>0$ be the smallest $\tau\in[0,1]$ 
such that either there exists $x'\in(-a,a)$ such that
$F_a^{\tau_1}(x')=0$ or $\partial_xF_a^{\tau_1}(a)=0$. In the latter case, we have 
$F_a^{\tau_1}(x)\ge 0$ for all $x\in[-a,a]$, 
hence $\partial_xF_a^{\tau_1}(a)=0$ would contradict the Hopf lemma. On the other hand,
the former situation would imply that the closest minimum of $F_a^{\tau_1}(x)$ to $(-a)$
is non-negative, which is also a contradiction. Thus, such $\tau_1$ can not exist, which
means that $F_a^\tau(x)>0$ for all $x\in(-a,a)$ and $\partial_xF_a^\tau(a)<0$ for all $\tau\in[0,1]$. 
As a consequence, by the same token, $F_a^\tau(x)$ can not attain a minimum on $[-a,a]$.
The only possibility to rule out then is that $F_a^\tau(x)$ would attain a single
local maximum on $[-a,a]$. On the other hand, that maximum would have to be larger
than $1$, and, as we have explained,  this is impossible. Now, the conclusion of 
Lemma~\ref{Fmonotone} follows.~\end{proof}

%
%
%
%
%

\subsubsection*{An a priori bound on the speed}

Now, we obtain a uniform bound on the speed $c_a^\tau$.
\begin{lemma}\label{cboundtau}
For any $\eps>0$ there exists $a_0>0$ such that
\begin{equation}\label{may2008}
-\eps<c_a^\tau< 2\sqrt{\kappa\alpha}+\eps \mbox{ for all } a>a_0~ 
\mbox{ and for all } \tau\in[0,1].
\end{equation}
\end{lemma}
\begin{proof}
As $0\le s_{a,\tau}^*(y)\le 1$ for all $y$, and $F_a^\tau(y)$ is monotonically decreasing, the function
$F_a^\tau(y)$ satisfies 
\begin{equation}\label{may2004}
-c_a^\tau \pdr{F_a^\tau}x-\kappa \pdrr{F_a^\tau}{x} \le 
\alpha F_a^\tau(x)(1-F_a^\tau(x))\le \alpha F_a^\tau(x),
\end{equation}
for all $\tau\in[0,1]$. 
%
On the other hand, the function $\psi^A(x) = Ae^{-\beta(x+a)}$ satisfies 
\begin{equation}\label{may2006}
-c_a^\tau \psi^A_x -\kappa \psi^A_{xx}\geq  \alpha \psi^A,
\end{equation}
as long as 
\begin{equation}
c_a^\tau \beta \geq \kappa\beta^2+{\alpha}.
\end{equation} 
%
Note that if $\beta>0$ and $A$ is sufficiently large, then $F_a^\tau(x)<\psi_A(x)$ for all
$x\in[-a,a]$. As we decrease $A$, we see from (\ref{may2004}) and (\ref{may2006}) that
$F_a^\tau(x)$ and $\psi^A(x)$ can not touch except at the boundary. Since $F_a^\tau(a)=0$, this
can only happen at $x=-a$, which means that $A=1$. It follows that
\[
F_a^\tau(x)\le e^{-\beta(x+a)} \hbox{ for all $-a\le x\le a$},
\]
and, in particular, we have $F_a^\tau(0)\le e^{-\beta a}$. This is a contradiction 
to (\ref{may1706}) if $\beta>\log 2/a$, and the upper 
bound for $c_a^\tau$ in (\ref{may2008}) follows.
%
%
%

For the lower bound we proceed in a similar way. Once again, monotonicity of $F_a^\tau(x)$
implies that 
\begin{equation}\label{may2004-bis}
-c_a^\tau \pdr{F_a^\tau}x-\kappa \pdrr{F_a^\tau}{x} \ge 0. 
\end{equation}
However, the function $\psi(x) =1-Be^{\beta(x-a)}$ satisfies
\begin{equation}\label{may2012}
-c_a^\tau\psi_x(x) -\kappa \psi_{xx}\leq 0, 
\end{equation}
provided that
\begin{equation}
c_a^\tau\beta  +\kappa\beta^2 \leq 0.
\end{equation}
Hence, if $c_a^\tau<0$, we can find $\beta>0$ such that (\ref{may2012}) holds.
As before, if $B>0$ is sufficiently large, we 
automatically have $F_a^\tau(x)>\psi(x)$. Decreasing $B$, we see that (\ref{may2004-bis})
and (\ref{may2012}) do not allow~$F_a^\tau(x)$ and $\psi(x)$ to touch inside $[-a,a]$,
and they can not intersect at $x=-a$ either. Thus, they touch at $x=a$ for the first time,
with $B=1$. It follows that 
\[
F_a^\tau(x)\ge  1-e^{\beta(x-a)}\hbox{ for all $x\in[-a,a]$},
\]
and, in particular, we have 
\[
\frac{1}{2}=F_a^\tau(0)>1-e^{-\beta a},
\]
which is a contradiction if $\beta>\log 2/a$, and the lower bound on $c_a^\tau$ in (\ref{may2008})
follows.
\end{proof}

\subsubsection*{A lower bound for $Q_a^\tau$}

We now obtain a series of bounds for the function $Q_a^\tau(x)$.
First, we establish a lower
bound on~$Q_a^\tau(x)$ and, in particular, show that it is positive. 
To this end, we need the following auxiliary lemma. 
Consider the eigenvalue problem
\begin{equation}\label{may2202}
c\psi'(x)-\kappa\psi''(x)=\mu_a(c)\psi,~~\hbox{$\psi(x)>0$ for all $-a<x<a$},
\end{equation}
with the boundary conditions 
\begin{equation}\label{may2204}
\psi'(-a)=0,~~\psi'(a)=\psi(a).
\end{equation}
Existence of such principal eigenfunction and eigenvalue follows from the standard
Sturm-Liouville theory -- see, for instance, Theorem~4.1 in~\cite{Hartman}. 
The next lemma gives a uniform bound on $\mu_a(c)$ as~$a\to+\infty$.
\begin{lemma}\label{lem-may2202}
For any $K>0$ there exists $C_K$ so that $|\mu_a(c)|\le C_K$ for all $|c|<K$.
\end{lemma}
\begin{proof} 
Writing 
\[
\psi(x)=\phi(x)\exp\Big(\farc{c}{2\kappa}x\Big)
\]
turns (\ref{may2202})-(\ref{may2204}) into 
\begin{equation}\label{may2718}
-\phi''(x)=-\gamma_a\phi,~~\hbox{$\phi(x)>0$ for all $-a<x<a$},
\end{equation}
with
\begin{equation}\label{may2728}
\gamma_a=-\farc{1}{\kappa}\Big(\mu_a(c)-\frac{c^2}{4\kappa}\Big)
\end{equation}
and with the boundary conditions
\begin{equation}\label{may2724}
\phi'(-a)=-\farc{c}{2\kappa}\phi(-a),~~\phi'(a)=\Big(1-\farc{c}{2\kappa}\Big)\phi(a).
\end{equation}
Note that if $\gamma_a<0$ then the eigenfunction is of the form
\[
\phi(x)=\cos(\sqrt{(-\gamma_a)}(x-z_a)),
\]
with some $z_a\in\Rm$. As $\phi(x)>0$ for all $x\in(-a,a)$, it follows 
that $\sqrt{(-\gamma_a)}\le \pi/(2a)$ in this case. 
As~$|c|\le K$, we conclude
that there exists $a_0>0$ so that for all $a>a_0$ if $\gamma_a\le 0$, then   
\begin{equation}\label{may2730}
|\mu_a(c)|\le \farc{K^2}{4\kappa}+1. 
\end{equation}
Let us now assume that $\gamma_a>0$ and set 
%
\begin{equation}\label{may2726}
r_1=-\farc{c}{2\kappa},~~r_2=1+r_1. 
\end{equation}
If $\gamma_a>0$, then the positive eigenfunction has the form 
\[
\eta(x)=\exp(\sqrt{\gamma_a}x)+\beta\exp(-\sqrt{\gamma_a}x).
\]
As we are only interested in bounds on $\gamma_a$, we may assume without loss
of generality that
\begin{equation}\label{may2732}
|\sqrt{\gamma_a}+r_1|>1,
\end{equation}
for otherwise $\gamma_a$ is automatically bounded, and thus so is 
$\mu_a(c)$. 
The boundary condition at $x=-a$
\begin{equation}\label{may2706}
\sqrt{\gamma_a}\exp(-\sqrt{\gamma_a}a)-
\beta\sqrt{\gamma_a}\exp(\sqrt{\gamma_a}a)=
r_1\exp(-\sqrt{\gamma_a}a)+
r_1\beta \exp(\sqrt{\gamma_a}a),
\end{equation}
implies that
\begin{equation}\label{ay2710}
\beta=\frac{\sqrt{\gamma_a}-r_1}{\sqrt{\gamma_a}+r_1}\exp(-2\sqrt{\gamma_a}a).
\end{equation}
Using this in the boundary condition at $x=a$
\begin{equation}\label{may2708}
\sqrt{\gamma_a}\exp(\sqrt{\gamma_a}a)-
\beta\sqrt{\gamma_a}\exp(-\sqrt{\gamma_a}a)=
r_2\exp(\sqrt{\gamma_a}a)+
r_2\beta \exp(-\sqrt{\gamma_a}a)
\end{equation}
gives
\begin{equation}\label{may2712}
\sqrt{\gamma_a}\Big(1-\frac{\sqrt{\gamma_a}-r_1}{\sqrt{\gamma_a}+r_1}\exp(-4\sqrt{\gamma_a}a)\Big)
=r_2\Big(1+\frac{\sqrt{\gamma_a}-r_1}{\sqrt{\gamma_a}+r_1}\exp(-4\sqrt{\gamma_a}a)\Big),
\end{equation}
so that
\begin{equation}\label{may2714}
\farc{r_2}{\sqrt{\gamma_a}}=
\frac{\sqrt{\gamma_a}+r_1-(\sqrt{\gamma_a}-r_1)
\exp(-4\sqrt{\gamma_a}a)}
{\sqrt{\gamma_a}+r_1+(\sqrt{\gamma_a}-r_1)\exp(-4\sqrt{\gamma_a}a)}
=1-\frac{2(\sqrt{\gamma_a}-r_1)
\exp(-4\sqrt{\gamma_a}a)}
{\sqrt{\gamma_a}+r_1+(\sqrt{\gamma_a}-r_1)\exp(-4\sqrt{\gamma_a}a)}.
\end{equation}
Let us assume that there exists a sequence $a_k\to +\infty$
such that 
\begin{equation}\label{may2736}
\sqrt{\gamma_{a_k}} \ge \farc{1}{\sqrt{a_k}}.
\end{equation}
 Then we have
\begin{equation}\label{may2734}
|\sqrt{\gamma_{a_k}}-r_1|\exp(-4\sqrt{\gamma_{a_k}}a_k)\le 
(\sqrt{\gamma_{a_k}}+|r_1|)\exp(-4\sqrt{\gamma_{a_k}}a_k)\le\farc{C}{a_k}+
|r_1|\exp(-4\sqrt{{a_k}}).
\end{equation}  
Passing to the limit $a_k\to+\infty$ in (\ref{may2714}) using (\ref{may2732})
and (\ref{may2734}) gives in that case
\begin{equation}\label{may2716}
\gamma_{a_k}\to r_2^2\hbox{ as $k\to+\infty$.}
\end{equation}
On the other hand, for any sequence $a_k\to+\infty$ for which (\ref{may2736}) does not
hold, we automatically have (\ref{may2730}).  
This finishes the proof.
\end{proof}

Now, we can prove the following lower bound on $Q_a^\tau(x)$.
\begin{lemma}\label{lem-may2204}
Let $\rho>C_K$, with $K=2\sqrt{\kappa\alpha}$, and $g(x)$ be the solution to
\begin{equation}\label{may2206}
(\rho-c_a^\tau)g(x)+c_a^\tau g'(x)-\kappa g''(x)=\tau e^x, 
\end{equation}
with the boundary conditions 
\begin{equation}\label{may2208}
g'(-a)=0,~~g'(a)=g(a),
\end{equation}
then $Q_a^\tau(x)\ge g(x)$ for all $x\in[-a,a]$.  
\end{lemma}
\begin{proof}
Recall that $Q_a^\tau(x)$ satisfies  
\begin{equation}\label{dec616}
(\rho-c^a)Q_a^\tau+c^a\pdr{Q_a^\tau}{x}-\kappa \pdrr{Q_a^\tau}{x}=\tau e^xH(R_a^\tau)\ge \tau e^x.
\end{equation}
Hence, the difference $f(x)=Q(x)-g(x)$ satisfies
\begin{equation}\label{may2210} 
(\rho-c^a)f(x)+c^af'(x)-\kappa f''(x)\ge 0,
\end{equation}
with the boundary conditions $g'(-a)=0$, $g'(a)=g(a)$. Lemmas~\ref{cboundtau}
and~\ref{lem-may2202}
imply that under the assumptions of the current lemma on the parameter
$\rho$, the principal eigenvalue
of the operator in the left side, with the boundary conditions (\ref{may2208}),
is positive, so that the comparison principe applies, thus $f(x)\ge 0$ for all $x\in[-a,a]$.
\end{proof}

As a consequence of Lemma~\ref{lem-may2204}, we have the following more explicit
lower bound.
\begin{lemma}\label{lem-may2304}
There exist $\rho_0>0$ and $a_0>0$ so that for $\rho>\rho_0$ and
$a>a_0$ the
function $Q_a^\tau(x)$ satisfies $Q_a^\tau(x)\ge \tau Ae^x$ for all~$x\in[-a,a]$,
and $\tau\in[0,1]$ and $A<{1}/({\rho-\kappa})$.
\end{lemma}
\begin{proof} 
An explicit solution to (\ref{may2206})-(\ref{may2208}) is
\begin{equation}\label{may2304}
g(x)=\tau z_1e^{\lambda_1x}+\tau z_2e^{-\lambda_2 x} +\frac{\tau}{\rho-\kappa}e^x,
\end{equation}
where 
\begin{equation}\label{may2316}
\lambda_1=\frac{c+\sqrt{c^2+4\kappa(\rho-c)}}{2\kappa}>0,~~
\lambda_2=\frac{-c+\sqrt{c^2+4\kappa(\rho-c)}}{2\kappa}>0,
\end{equation}
and the constants $z_1$ and $z_2$ are given by
\begin{equation}\label{dec820}
z_1=\frac{e^{-a}}{\rho-\kappa}\Big{(} 
 \frac{\lambda_2(\lambda_1-1)}{\lambda_2+1}e^{(\lambda_1+2\lambda_2)a}
 -\lambda_1e^{-\lambda_1a} \Big{)}^{-1},
~
z_2=\frac{e^{-a}}{\rho-\kappa}\Big{(}\lambda_2e^{\lambda_2a} - 
\frac{\lambda_1(\lambda_2+1)}{\lambda_1-1}e^{-(\lambda_2+2\lambda_1)a}\Big{)}^{-1}.
\end{equation}
Note that $\lambda_1>0$ and $\lambda_2>0$ for $\rho$ sufficiently large, 
and for $a>a_0$ sufficiently large we also have that 
both $z_1>0$ and $z_2>0$, and the conclusion
of the present Lemma follows from Lemma~\ref{lem-may2204}.

\end{proof}

\subsubsection*{The monotonicity  of $Q_a^\tau$ and $R^\tau_a$}

Next, the uniform bound on the speed~$c^a$ and positivity of $Q_a^\tau(x)$
allow us to 
show monotonicity of $Q_a^\tau(x)$ and $R^\tau_a(x)$.   
\begin{lemma}\label{QRmonTauV}
There exists $a_0>0$ so that
function $Q^\tau_a(x)$ is increasing in $x$ and the functions~$R^\tau_a(x)$ 
and $s_{a,\tau}^*(x)$ are decreasing in $x$ for all $a>a_0$. 
\end{lemma}
\begin{proof}
Note that if $Q_a^\tau(x)$ 
is increasing in~$x$, then, as
\begin{equation}\label{may2002}
R_a^\tau(x)=1-\tau+\frac{\alpha\tau}{2} e^{-x}  \int_{x}^a [Q_a^\tau(y)-Q_a^\tau(x)]
(-\pdr{F_a^\tau(y)}{y})dy=
1-\tau+\frac{\alpha\tau}{2} e^{-x}  \int_{x}^a \pdr{Q_a^\tau(y)}{y}F_a^\tau(y)dy,
\end{equation}
Lemma \ref{Fmonotone} implies that  $R_a^\tau(x)$ is decreasing in $x$. In addition,
monotonicity of $R_a^\tau(x)$ implies monotonicity of $s_{a,\tau}^*(x)$, hence 
we only need to study monotonicity of~$Q_a^\tau(x)$.  
Differentiating (\ref{may1708}) shows that 
\begin{equation}\label{derQ}
(\rho-c_a^\tau)Q'+c_a^\tau Q'_x -\kappa Q'_{xx} = \tau e^xH(R_a^\tau)+\tau e^xH'(R_a^\tau)
\pdr{R_a^\tau}{x},
\end{equation}
with $Q'(x)= \partial_xQ_a^\tau(x)$, and 
from (\ref{may2002}) we see that
\begin{equation}
\pdr{R_a^\tau(x)}{x} = -R_a^\tau(x)+1-\tau-\frac{\tau\alpha}{2}e^{-x}\pdr{Q_a^\tau(x)}{x}F_a^\tau(x).
\end{equation}
Recalling that $H(R)$ is given explicitly by (\ref{Hdef2}), we now write (\ref{derQ}) as
\begin{equation}\label{may2224}
\begin{aligned}
(\rho-c_a^\tau)Q'+c_a^\tau Q'_x -\kappa Q'_{xx} = 
& \tau e^x\begin{cases}
2-2\tau- \tau \alpha e^{-x}Q'(x)F_a^\tau(x),~\mbox{ if } R_a^\tau>1,\\
1-(R_a^\tau)^2+2R_a^\tau(1-\tau) - \tau R_a^\tau\alpha e^{-x}Q'(x)F_a^\tau(x),~ 
\mbox{ if } 0\leq R_a^\tau\leq 1,\\
1,~ \mbox{ if } R_a^\tau<0.
\end{cases}
\end{aligned}
\end{equation}
It follows that
\begin{align*}
 &-\kappa Q'_{xx}+c_a^\tau Q'_x +   (\rho-c_a^\tau)Q'(x)  
 + \tau\alpha F_a^\tau(x)Q'(x)S^*(R_a^\tau(x))\ge 0.
\end{align*}
Assumption $\rho>\rho_0$
in Theorem~\ref{2.1} together with Lemma~\ref{cboundtau} implies that 
$c_a^\tau<\rho$ for $a>a_0$ if $\rho_0$ is sufficiently large.
It follows that $Q'(x)$ can not attain an interior negative minimum. We also
have~$Q'(-a)=0$ and~$Q'(a)=Q_a^\tau(a)>0$, thus a negative minimum of $Q'(x)$
can not be attained at~$x=\pm a$ either. 
Therefore, we have $Q'(x)\geq 0$
for all $x\in[-a,a]$ and $Q_a^\tau(x)$ is increasing in $x$. 
\end{proof}



 %
%
%

\subsubsection*{An upper bound for $Q_a^\tau(x)$}

In this section we show  that $Q_a^\tau(x)$ is bounded from above. 
\begin{lemma}\label{lem-dec602}
There exist $a_0>0$, $\rho_0>0$ and $C>0$ so that if $\rho>\rho_0$ then
\begin{equation}\label{dec614}
Q_a^\tau(x)\le \farc{C}{\rho}(1+e^x),\hbox{ for all $a>a_0$, $-a<x<a$ and $\tau\in[0,1]$.}
\end{equation}
\end{lemma}
\begin{proof}
First, note that $Q_a^0\equiv 0$ trivially satisfies (\ref{dec614}). Our goal will be to show that
if we choose~$K$ sufficiently large, then (\ref{dec614}) with $C=K$
can not be violated for any $\tau\in[0,1]$.
To this end, assume that $\tau_1>0$ is the smallest $\tau>0$ such that there exists $x_1\in[-a,a]$ 
such that 
\begin{equation}\label{dec620}
Q_a^\tau(x_1)=K(1+e^{x_1}).
\end{equation}
Then, we still have
\begin{equation}\label{dec622}
Q_a^{\tau_1}(x)\le K(1+e^x),\hbox{ for all $-a<x<a$.}
\end{equation}
As $H(r)\le 1+2r$, it follows from (\ref{dec616}) that
\begin{equation}\label{dec618}
\begin{aligned}
&(\rho-c_a^{\tau_1})Q_a^{\tau_1}+c_a^{\tau_1}\pdr{Q_a^{\tau_1}}{x}-\kappa \pdrr{Q_a^{\tau_1}}{x}
\le \tau_1 e^x(1+2R_a^{\tau_1}(x))\\
&\le 
\tau_1 e^x(1+2(1-\tau_1))+\tau_1^2\alpha\int_{x}^a\pdr{Q_a^{\tau_1}(y)}{y}{F_a^{\tau_1}(y)}dy.
\end{aligned}
\end{equation}
\begin{proof}
To estimate the integral in the right side of (\ref{dec618}), we will use the following lemma.
\begin{lemma}\label{BLemma}
There exist constants $B$ and $a_0$, such that 
\begin{equation}\label{may2340}
\int_{-a}^a (Q^\tau_a)'(x)F^\tau_a(x) dx\leq B
\end{equation}
for all $a>0$ and $\tau$.
\end{lemma}
We first note that
\begin{equation}\label{dec624}
\begin{aligned}
\int_{-a}^0\pdr{Q_a^{\tau_1}(y)}{y}{F_a^{\tau_1}(y)}dy&\le \farc{1}{2}Q_a^{\tau_1}(0)+
\int_{-a}^0{Q_a^{\tau_1}(y)}\big(-\pdr{F_a^{\tau_1}(y)}{y}
\big)dy\\
&\le K+ 2K\int_{-a}^0\big(-\pdr{F_a^{\tau_1}(y)}{y}
\big)dy=2K,
\end{aligned}
\end{equation}
because of (\ref{dec622}) and normalization (\ref{may1706}). To estimate the integral from $0$
to $a$, recall that $x_0$, as defined in (\ref{may2226}) is $x_0=\sup\{x: R_a^{\tau_1}(x)>1\}$. 
If $x_0\le 0$ we have  
\begin{equation}\label{dec634}
\int_0^a \pdr{Q_a^{\tau_1}(x)}{x}F_a^{\tau_1}(x) dx\le 
\int_{x_0}^a\pdr{Q_a^{\tau_1}(x)}{x}F_a^{\tau_1}(x)dx =\frac{2}{\alpha}e^{x_0}R(x_0)
\leq \frac{2}{\alpha}.
\end{equation}
The case $x_0\ge 0$ is handled by the following
upper bound on $F_a^\tau$.
\begin{lemma}\label{dratetau}
There exists $\alpha_0>0$ and a constant $C$, 
independent of $a$ and $\tau$, such that for all~$\alpha>\alpha_0$
if $x_0>0$ then
\begin{equation}\label{may2234}
F_a^\tau(x) \leq Ce^{-2x}.
\end{equation}
\end{lemma}
\begin{proof}
Note that (\ref{may2234}) holds automatically for $x<0$ if $C>1$, 
since $F_a^\tau(x)<1$, thus we may
assume without loss of generality that $x>0$. 
Since $x_0>0$ we have $s_{a,\tau}^*(x)=1$ for~$x<0$. Therefore, we have for all $x>0$  
\begin{equation}
\int_{-a}^x s_{a,\tau}^*(y)(-\pdr{F_a^\tau(y)}{y})dy\geq 
\int_{-a}^0 s_{a,\tau}^*(y)(-\pdr{F_a^\tau(y)}{y})dy =\frac{1}{2},
\end{equation}
due to (\ref{may1706}). 
It follows that for $x>0$ we have
\begin{equation}
\begin{aligned}
\int_{-a}^x[(1-\tau )+\tau s_{a,\tau}^*(y)](-\pdr{F_a^\tau(y)}{y}) dy&=  (1-\tau)(1-F_a^\tau(x))+ 
\tau\int_{-a}^x s_{a,\tau}^*(y) )(-\pdr{F_a^\tau(y)}{y})dy\\
\geq&  \frac{(1-\tau)}{2}  +\frac{ \tau}{2} =\frac{1}{2}. 
\end{aligned}
\end{equation}
Therefore, for $x>0$ we have 
\begin{equation}\label{dec802}
-c_a^\tau\pdr{F_a^\tau}{x} \geq \kappa \pdrr{F_a^\tau}{x} + \frac{\alpha}{2} F_a^\tau. 
\end{equation}
We integrate (\ref{dec802}) from $x$ to $y$, with $0<x<y$, to get
\begin{equation}\label{dec804}
c_a^\tau F_a^\tau(x)-c_a^\tau F_a^\tau(y)\geq \kappa \pdr{F_a^\tau(y)}{x}-\kappa \pdr{F_a^\tau(x)}{x}
+ \frac{\alpha}{2} 
\int_x^y F_a^\tau(\xi)d\xi.
\end{equation}
Next, integrate (\ref{dec804}) in $y$ from $z$ to $z+1$, with $z>x$, to get
\begin{equation}\label{may2230}
c_a^\tau F_a^\tau(x)-c_a^\tau\int_z^{z+1}F_a^\tau(y) dy \geq \kappa F_a^\tau(z+1)-
\kappa F_a^\tau(z)-\kappa \pdr{F_a^\tau(x)}{x}+ 
\frac{\alpha}{2} 
\int_z^{z+1}\int_x^y F_a^\tau(\xi) d\xi dy.
\end{equation}
We can estimate the left side of (\ref{may2230}) simply as
\begin{equation}\label{dec810}
c_a^\tau F_a^\tau (x)-c_a^\tau\int_z^{z+1}F_a^\tau(y) dy\leq c_a^\tau F_a^\tau(x).
\end{equation}
For the right side of (\ref{may2230}) we have, as $F_a^\tau$ is decreasing: 
\begin{equation}\label{may2232}
\begin{aligned}
& \kappa  F_a^\tau(z+1)-\kappa F_a^\tau(z)-\kappa \pdr{F_a^\tau(x)}{x}+ 
\frac{\alpha}{2} 
\int_z^{z+1}\int_x^y F_a^\tau(\xi) d\xi dy\\
&\geq \kappa F_a^\tau(z+1)-\kappa F_a^\tau(z)-\kappa \pdr{F_a^\tau(x)}{x}+ 
\frac{\alpha}{2} 
\int_z^{z+1} (y-x) F_a^\tau(y)  dy\\
&\geq \kappa F_a^\tau(z+1)-\kappa F_a^\tau(z)-\kappa \pdr{F_a^\tau(x)}{x}+ 
\frac{\alpha}{2} 
F_a^\tau(z+1)\int_z^{z+1}(y-x) dy\\
&\geq \kappa F_a^\tau(z+1)-\kappa F_a^\tau(z) 
+ 
\frac{\alpha}{2}(z-x)
F_a^\tau(z+1) .
\end{aligned}
\end{equation}
Putting (\ref{may2230}), (\ref{dec810}) and (\ref{may2232}) together, we get, for all $z>x$:  
\begin{equation}
c_a^\tau F_a^\tau(x)\geq  \kappa F_a^\tau(z+1)-\kappa F_a^\tau(z) 
+ 
\frac{\alpha}{2}(z-x) 
F_a^\tau(z+1) .
\end{equation}
Adding $\kappa F_a^\tau(x)$ to both sides gives, as $x<z$:
\begin{equation}
\begin{aligned}
(c_a^\tau+\kappa)F_a^\tau(x)&\geq \kappa F_a^\tau(z+1)+\kappa F_a^\tau(x)-\kappa F_a^\tau(z)
+ \frac{\alpha}{2}(z-x) 
F_a^\tau(z+1) \\
&\ge \kappa F_a^\tau(z+1) 
+ \frac{\alpha}{2} (z-x)
F_a^\tau(z+1) 
\end{aligned}
\end{equation}
%
Taking $z=x+1$ leads to
\begin{equation}
(c_a^\tau+\kappa)F_a^\tau(x)\geq \Big(\kappa+\frac{\alpha}{2}\Big)F_a^\tau(x+2) ,
 \end{equation}
thus 
\begin{equation}
F_a^\tau(x+2)\leq \frac{c_a^\tau+\kappa}{\kappa+ \alpha/2} 
F_a^\tau(x).
\end{equation}
Now, (\ref{may2234}) follows for $x>0$ if we take $\alpha$ sufficiently large, 
as $|c_a^\tau|\le 2\sqrt{\kappa\alpha}$ by Lemma~\ref{cboundtau}.  
\end{proof}
We now go back to the proof of Lemma~\ref{BLemma}. It follows from Lemma~\ref{dratetau}
that if $x_0\ge 0$ then
\begin{equation}\label{dec630}
\begin{aligned}
\int_0^a \pdr{Q_a^{\tau_1}(x)}{x}F_a^{\tau_1}(x) dx\le 
C\int_{0}^a\pdr{Q_a^{\tau_1}(x)}{x}e^{-2x}dx \le C Q_a^{\tau_1}(a)e^{-2a}+
2C\int_0^a Q_a^{\tau_1}(x)e^{-2x}dx\le CK,
\end{aligned}
\end{equation}
because of (\ref{dec622}). Together with (\ref{dec624}), this gives (\ref{may2340}). 
\end{proof}

We continue the the proof of Lemma~\ref{lem-dec602}.
Using (\ref{dec624}), (\ref{dec634}) and (\ref{dec630}) in (\ref{dec618})
gives
\begin{equation}\label{dec632}
\begin{aligned}
&(\rho-c_a^{\tau_1})Q_a^{\tau_1}+c_a^{\tau_1}\pdr{Q_a^{\tau_1}}{x}-\kappa \pdrr{Q_a^{\tau_1}}{x}
\le \tau_1 e^x(1+2(1-\tau_1))+C\tau_1K,
\end{aligned}
\end{equation}
with the boundary conditions  
\begin{equation}\label{dec636}
\pdr{Q_a^{\tau_1}}{x}(-a)=0,~~\pdr{Q_a^{\tau_1}}{x}(-a)=Q_a^{\tau_1}(a).
\end{equation}
The comparison principle implies that 
\begin{equation}\label{dec638}
Q_a^\tau(a)\le (1+2(1-\tau_1))g(x)+\farc{C\tau_1 K}{\rho-c_a^{\tau_1}}(1+q(x)),
\end{equation}
where $g(x)$ is the explicit solution to (\ref{may2206})-(\ref{may2208})  given by (\ref{may2304}),
and $q(x)$ is the solution to 
\begin{equation}\label{dec640}
\begin{aligned}
&(\rho-c_a^{\tau_1})q+c_a^{\tau_1}\pdr{q}{x}-\kappa \pdrr{q}{x}
= 0,
\end{aligned}
\end{equation}
with the boundary conditions  
\begin{equation}\label{dec636bis}
\pdr{q}{x}(-a)=0,~~\pdr{q}{x}(a)=q(a)-1.
\end{equation}
The function $q(x)$ has the form
\begin{equation}\label{dec638bis}
q(x)=\mu_1e^{\lambda_1(x-a)}+\mu_2e^{-\lambda_2 (x+a)},
\end{equation}
with $\lambda_{1,2}>0$ given by (\ref{may2316}), and the coefficients $\mu_{1,2}$ given by
\begin{equation}\label{dec812}
\begin{aligned}
&\mu_1=\farc{\lambda_2}{\lambda_1}e^{2\lambda_1a}\mu_2,\\
&\mu_2=\Big[\lambda_2\Big(\farc{1}{\lambda_1}-1\Big)e^{2\lambda_1a}
-(\lambda_2-1)e^{-2\lambda_2a}\Big]^{-1},
\end{aligned}
\end{equation}
so that for $a$ large we have
\begin{equation}\label{dec814}
\mu_1\approx -\farc{1}{\lambda_1-1},~~
\mu_2\approx -\farc{\lambda_1}{\lambda_2(\lambda_1-1)}e^{-2\lambda_1a}.
\end{equation}
It follows from (\ref{may2316}) that $\lambda_1>1$ for $\rho>\rho_0$, due to the bounds on $c_a\tau$
in Lemma~\ref{cboundtau}, hence~$\mu_{1,2}<0$ for $\rho>\rho_0$. We conclude from 
(\ref{dec638}) and (\ref{dec638bis}) that
\begin{equation}\label{dec816}
Q_a^\tau(a)\le (1+2(1-\tau_1))g(x)+\farc{C\tau_1 K}{\rho-c_a^{\tau_1}}\le 
3g(x)+\farc{C\tau_1 K}{\rho},
\end{equation}
for $\rho>\rho_0$. Going back to the explicit expression (\ref{may2304}) for $g(x)$ and using
(\ref{may2316}) and (\ref{dec820}), we see that
\begin{equation}\label{dec818}
g(x)\le \farc{C}{\rho}(1+e^x),
\end{equation} 
for $a>a_0$ and $\rho>\rho_0$. It follows from (\ref{dec816}) that 
\begin{equation}\label{dec822bis}
Q_a^\tau(a)\le  \frac {C}{\rho}e^x+\farc{CK}{\rho},\hbox{ for all $-a<x<a$.}
\end{equation}
This gives a contradiction to (\ref{dec620}) if $K>C/\rho$, finishing the proof of Lemma~\ref{lem-dec602}.
\end{proof}

\subsubsection*{An upper bound for $\partial_xQ_a^\tau(x)$} 

Next, we show that the first derivative of $Q_a^\tau$ is bounded from above 
for all $\tau$.
\begin{lemma}\label{QprBoundedtau}
There exist $a_0>0$ and $\rho_0>0$, and $C>0$ so that
$Q'(x)=(Q_a^\tau)'(x)$ satisfies 
\begin{equation}\label{may2338}
Q'(x)\le Ce^x{+C e^{-\lambda_2a}}
\end{equation}
for all $a>a_0$, $x\in[-a,a]$ 
and all $\tau\in[0,1]$, with $\lambda_2>0$ as in (\ref{may2316}).
\end{lemma}
\begin{proof}
As we have already shown that $Q'(x)\ge 0$, it follows from (\ref{may2224}) that
\begin{equation}\label{dec822}
-\kappa Q'_{xx}+c_a^\tau Q'_x +  (\rho-c_a^\tau) Q'(x)\le 2e^x,
\end{equation}
with the boundary condition $Q'(-a)=0$. To get the boundary condition at $x=a$ we use (\ref{may1708})
and Lemma~\ref{lem-dec602}:
\begin{equation}\label{dec824}
\kappa\partial_x Q'(a)=(\rho-c_a^\tau)Q_a^\tau(a)+c_a^\tau Q'(a)-\tau e^a H(R_a^\tau(a))\le
\rho Q_a^\tau(a)\le C e^a.
\end{equation}
If $\rho>\rho_0$,  the function 
\begin{equation}\label{dec826}
p(x)=
\farc{2C}{\kappa}e^x,
\end{equation}
satisfies
\begin{equation}\label{dec828}
-\kappa p_{xx}+c_a^\tau p_x +  (\rho-c_a^\tau) p(x)\ge 2e^x,
\end{equation}
with the boundary conditions
\begin{equation}\label{dec830}
\begin{aligned}
&p(-a)> 0,~~
\kappa p_x(a)=
2Ce^a\ge Ce^a.
\end{aligned}
\end{equation}
Then, the difference $\xi(x)=p(x)-Q_a^\tau(x)$ satisfies
\begin{equation}\label{dec834}
\begin{aligned}
&-\kappa \xi_{xx}+c_a^\tau \xi_x +  (\rho-c_a^\tau) \xi(x)\ge 0,\\
&\xi(-a)>0,~~\kappa\xi_x(a)>0.
\end{aligned}
\end{equation}
It follows that if $\rho>\rho_0$ then
$\xi(x)$ can not attain a negative minimum inside $(-a,a)$ and, in addition, it can not  attain
a minimum at $x=a$. Thus, $\xi(x)>0$ for all $x\in(-a,a)$, and 
\begin{equation}\label{dec836}
Q'(x)\le p(x)=\farc{2C}{\kappa}e^x \hbox{ for all $x\in(-a,a)$,}
\end{equation}
and the proof of Lemma~\ref{QprBoundedtau} is complete.
\end{proof}

\subsubsection*{A uniform gradient bound for $F_a^\tau(x)$}

We first obtain a uniform bound for the derivative of $F_a^\tau(x)$. 
To simplify the notation, we drop the subscripts $a$ and $\tau$.
\begin{lemma}
There exists $a_0>0$ and $C>0$ such that 
\begin{equation}\label{may2220}
\int_{-a}^a |F_x|^2 dx \leq C \mbox{ for all } a>a_0~\mbox{ and for all } \tau\in[0,1].
\end{equation}
\end{lemma}
\begin{proof}
Integrating (\ref{taueq}) from $-a$ to $a$ gives
\begin{equation}\label{may2214}
c=\kappa F_x(a)-\kappa F_x(-a)+\alpha(1-\tau)\int_{-a}^a F(1-F)dx+\alpha\tau
\int_{-a}^a F(x)\int_{-a}^{x}s^*(y)(-F_y(y))dydx.
\end{equation}
We also multiply both sides of (\ref{tausis}) by $F$, integrate from $-a$ to $a$,
and use (\ref{may2214}):
\begin{equation}\label{may2216}
\begin{aligned}
\frac{c}{2}&+ \kappa \int_{-a}^a F_x^2 dx+\kappa F_x(-a)
=\alpha(1-\tau)\int_{-a}^a F^2(1-F)dx+\alpha\tau
\int_{-a}^aF^2(x)\int_{-a}^{x}s^*(y)(-F_y(y))dydx\\
&\le c-\kappa F_x(a)+\kappa F_x(-a),
\end{aligned}
\end{equation}
so that
\begin{equation}\label{may2218}
\kappa \int_{-a}^a F_x^2 dx\le \farc{c}{2}-\kappa F_x(a).
\end{equation}
The uniform bounds on $c$ in Lemma~\ref{cboundtau} allow to apply the standard
elliptic regularity results to conclude that $|F_x(a)|\le C$, with $C$ that does not depend
on $a$, and (\ref{may2220}) follows.
\end{proof}

\subsubsection{The degree argument}

We have by now proved the a priori bounds in 
Proposition~\ref{FiniteInterval}.  
We now use these a priori bounds to finish the proof of the existence part of
Proposition~\ref{FiniteInterval} using a Leray-Schauder degree argument. 
Let us define the map $\mathcal{L}_\tau (c,F,Q)=(\theta,G,T)$ as the solution operator for the system
\begin{equation}\label{may2344}
\begin{aligned}
&-cG_x=\kappa G_{xx}+\int_{-a}^x[(1-\tau )+\tau s^*(y)](-F_y(y))dy\\
&(\rho -c)T+cT_x-\kappa T_{xx}=\tau e^x H(R(x))
\end{aligned}
\end{equation}
with the boundary conditions
\[
G(-a)=1, ~G(a)=0, ~T_x(-a)=0~\mbox{ and }T_x(a)=T(a),
\]
and with
\begin{equation}
R(x)=1-\tau+\tau\frac{\alpha}{2}e^x\int_x^a[Q(y)-Q(x)](-F^\tau_y(y))dy
\end{equation}
and $s^*(x)=\min\{1,R(x)\}$. 
The constant $\theta$ is defined as
\begin{equation}
\theta=\frac{1}{2}-\max_{x\in[0,a]}F(x) +c.
\end{equation}
 
This operator maps the Banach space 
$X=\R\times C^1([-a,a])\times C^1([-a,a])$ with the norm 
\[
\|c,F,Q\|_X=\max\{|c|,\|F\|_{C^1}, \|Q|\|_{C^1} \},
\]
to itself, and
its fixed points are solutions to (\ref{tausis}). Therefore, it suffices 
to show that the operator~$\mathcal{F}_\tau=\hbox{Id}-\mathcal{L_\tau}$ has 
a nontrivial kernel for all $\tau\in[0,1]$.  Let $B_M$ be a ball of radius~$M$ in~$X$
centered at the origin. Using the a priori bounds obtained above we can choose $M$ 
sufficiently large to ensure that $\mathcal{F}_\tau$ does not vanish on the 
boundary $\partial B_M$. As the Leray-Schauder degree is homotopy invariant, 
it is enough to show that $\deg(\mathcal{F}_0,B_M,0)\neq 0$. 
Note  that
\begin{equation}
\mathcal{F}_0(c,F,Q)=\Big(\max_{x>0}F^c_0(x)-\frac{1}{2},F-F_0^c,Q\Big),
\end{equation}
where $F_0^c$ solves
\begin{equation}
-cF_0'=\kappa F_0''+\alpha F(1-F),~~F_0(-\infty)=1,~F_0(+\infty)=0.
\end{equation}
Hence, $\deg(\mathcal{F}_0,B_M,0)=-1$, thus 
$\mathcal{F}_\tau$ has a nontrivial kernel. Therefore, a solution  
to (\ref{tausis}) exists for all $\tau\in[0,1]$, which proves the existence part of 
Proposition \ref{FiniteInterval}.  


\subsection{Identification of the limit}\label{Limit}
To complete the proof of Proposition \ref{FiniteInterval} we get  uniform bounds on the transition point $x_0^a$.
%
\begin{lemma}\label{lem-may2708}
There exist constants $x_0^+$, $x_0^-$ and $a_0$ such that for all
$a>a_0$ we have 
\[
x_0^-\leq x_0^a\leq x_0^+.
\]
\end{lemma}
\begin{proof}
Recall that the point $x_0^a$ is determined by $R(x_0^a)=1$, so that
\begin{equation}\label{feb414}
1 =\frac{\alpha}{2}e^{-x_0^a}\int_{x_0^a}^a Q^a_y(y)F^a(y) dy.
\end{equation}
Using 
Lemma \ref{BLemma}, we obtain
\[
\exp\{-x_0^a\}\geq \frac{2}{\alpha B},
\]
hence
\begin{equation}\label{may2740}
x_0^a\leq  \log  ({\alpha B}).
\end{equation} 
For a lower bound on $x_0^a$, let us 
assume that $x_0^a<0$, 
and write, for any $z>0$:
\begin{equation}
\begin{aligned}
 \frac{2}{\alpha} &=e^{-x_0^a}\int_{x_0^a}^a[Q^a(y)-Q^a(x_0^a)](-F_y^a(y)) dy  
 > e^{-x_0^a}\int_0^a [Q^a(y)-Q^a(x_0^a)](-F_y^a(y)) dy 
 \\
 &> e^{-x_0^a}\int_0^a [Q^a(y)-Q^a(0)](-F_y^a(y)) dy> 
 e^{-x_0^a}\int_z^a [Q^a(y)-Q^a(0)](-F_y^a(y)) dy\\
 &\ge e^{-x_0^a}\int_z^a [Ae^y-B](-F_y^a(y)) dy.
\end{aligned}
\end{equation}
We used Lemma~\ref{lem-may2304} in the last step above to bound $Q^a(y)$ from 
below and Lemma~\ref{lem-dec602} to bound~$Q^a(0)$ from above. We may now choose
$z>0$ so that $Ae^y>2B$ for all $y>z$, so that
\begin{equation}
\begin{aligned}
 \frac{2}{\alpha}   &\ge Be^{-x_0^a}\int_z^a  (-F_y^a(y)) dy=Be^{-x_0^a}F^a(z).
\end{aligned}
\end{equation}
As $z$ does not depend on $a$, and $F^a(0)=1/2$, the Harnack inequality implies that
there exists~$s>0$ that does not depend on $a$ so that $F^a(z)>s$, so that
\[
e^{x_0^a}>\farc{\alpha Bs}{2},
\]
finishing the proof of the lower bound for $x_0^a$.
\end{proof}

The a priori bounds obtained in Proposition \ref{FiniteInterval}
allows us to extract a subsequence $a_n\to+\infty$ such that the 
corresponding sequence $(c^{a_n},F^{a_n},Q^{a_n})$ converges
to a limit $(c,F,Q)$, in $C^{2,\alpha}_{loc}(\Rm)$. Moreover,
the functions $F$ and $Q$ are monotonic. 
The upper bound
on $Q_a'(x)$ in Lemma~\ref{QprBoundedtau} and the upper bound on $F_a(x)$
in Lemma~\ref{lem-may2708} imply that
\begin{equation}\label{may2746}
\int_{x}^{a_n}Q_{a_n}'(y)F_{a_n}(y)dy\to \int_{x}^{\infty}Q'(y)F(y)dy,
\end{equation}
hence the corresponding
sequences $s_{a_n}^*(x)$ and $R_{a_n}(x)$ converge as well to their respective
limits~$s^*(x)$ and $R(x)$ such that $R(x)\ge 0$ and $s^*(x)=\min(1,R(x))$,
and
\begin{equation}\label{may2824}
R(x)=\farc{\alpha}{2}e^{-x}\int_x^\infty Q'(y)F(y)dy.
\end{equation}
In particular,
as a consequence, the function $Q$ satisfies the second equation in (\ref{may1502}):
\begin{equation}\label{may2748}
\rho  Q=c Q-c \frac{\partial Q}{\partial x}+ 
\kappa\frac{\partial^2 Q}{\partial x^2}+e^{x} H(R).
\end{equation}
In order to
see that $F(x)$ satisfies the first equation in (\ref{may1502}), let us take $x_0^\pm$ 
as in Lemma~\ref{lem-may2708} and write
\begin{equation}\label{may2742}
\begin{aligned}
\int_{-a_n}^{x}s_{a_n}^*(y)(-F_y^{a_n}(y))dy&=\int_{-a_n}^{x_0^-}(-F_y^{a_n}(y))dy+
\int_{x_0^-}^{x_0^+}s_{a_n}^*(y)(-F_y^{a_n}(y))dy\\
&+\int_{x_0^+}^{a_n}s_{a_n}^*(y)(-F_y^{a_n}(y))dy
=1-F^{a_n}(x_0^-)+I_n+II_n.
\end{aligned}
\end{equation}
The bounded convergence theorem implies that
\begin{equation}\label{may2744}
I_n\to \int_{x_0^-}^{x_0^+}s^*(y)(-F_y(y))dy,
\end{equation}
while the Lebesgue dominated convergence theorem and (\ref{may2234}) imply that
\begin{equation}
II_n\to 
\int_{x_0^+}^{\infty}s^*(y)(-F_y(y))dy.
\end{equation}
It follows 
that~$F$ satisfies  
\begin{equation}\label{may2750}
-cF_x-\kappa F_{xx}=\alpha F^a(x)\int_{-a}^{x}s_a^*(y)(-F_y^a(y))dy.
\end{equation}
%
%
%
%
It remains to show that the limit $F$ satisfies the correct
boundary conditions and that $Q(x)$ converges to a positive constant on 
the left and grows exponentially on the right, as in (\ref{may1504}).
Note that Lemma~\ref{dratetau} implies that $F(x)\to 0$ as $x\to+\infty$. 
The next lemma takes care of the limit on the left. 
\begin{lemma}
The limiting  function $F(x)$ converges to $1$ as $x\rightarrow -\infty$.
\end{lemma}
\begin{proof}
Let $x^-_0$ be the lower bound on $x_0^a$ as in Lemma \ref{lem-may2708}. Then,
for $a>a_0$  the function $F^a$ satisfies  
\[
-c^aF^a_x-\kappa F_{xx}^a=\alpha F^a\int_{-a}^x(-F_y^a(y))dy=\alpha F^a(1-F^a),
~~\hbox{ for $x<x_0^-$.}
\]
Integrating both sides from $(-a)$ to $x_0^-$ gives
\[
c^a(1-F^a(x_0^-))-\kappa F_x^a(-a)+\kappa F_x^a(x_0^-)=
\int_{-a}^{x_0^-}\alpha F^a(1-F^a) dx.
\]
Note that  $c^a$  is bounded by Lemma \ref{cboundtau}, $F^a(x)$ 
is bounded for all $x$ and $F_x(-a)$ and $F_x(x_0^-)$ are also bounded by 
elliptic regularity. Therefore, the left side is bounded independently of
$a$ for $a>a_0$, hence so is the integral in the right side.
It follows that the integral
\[
\int_{-\infty}^{x_0^-}F(1-F)dx
\]
is finite. 
As $F(x)$ is monotonically decreasing, and $F(x)\ge 1/2$ for $x\le 0$, it follows
that $F(x)\to 1$ as~$x\to-\infty$.
\end{proof}

Next, we look at the left limit of $Q(x)$.
\begin{lemma}\label{lem-may2804}
The limiting  function $Q(x)$ converges to a positive constant 
$q_-$ as $x\rightarrow -\infty$.
\end{lemma}
\begin{proof}
Note that for $x<x_0^-$ we have $R(x)>1$, so that $H(R(x))=2R(x)$, and
(\ref{may2748}) becomes 
\begin{equation}\label{may2820}
\rho  Q(x)=c Q(x)-c \frac{\partial Q(x)}{\partial x}+ 
\kappa\frac{\partial^2 Q(x)}{\partial x^2}+\alpha \int_x^\infty Q_y(y)F(y)dy.
\end{equation}
As the function $Q(x)$ is monotonically increasing, and the derivatives $Q'(x)$
and $Q''(x)$ are uniformly bounded for $x<0$, there exists a sequence $x_n\to-\infty$
such that both~$Q'(x_n)\to 0$ and~$Q''(x_n)\to 0$. Passing to the limit $n\to+\infty$
in (\ref{may2820}) leads to 
\begin{equation}\label{may2826}
(\rho -c)q_-=\alpha \int_{-\infty}^\infty Q_y(y)F(y)dy,
\end{equation}
where 
\[
q_-=\lim_{x\to-\infty}Q(x).
\]
It follows that $q_->0$.
\end{proof}

Finally, we look at the behavior of $Q(x)$ on the right.
\begin{lemma}
The limit
\begin{equation}\label{may2828}
\lim_{x\to+\infty} Q(x)e^{-x}
\end{equation}
exists and equals $1/(\rho-\kappa)$. 
\end{lemma}
\begin{proof}
Lemmas~\ref{lem-may2304} and \ref{lem-dec602} 
imply that there exist $0<A_1<A_2$ and $B>0$ such that
\begin{equation}\label{may2756}
A_1e^{x}\leq Q(x)\leq A_2e^{x}+B.
\end{equation}
Consider the function $Z(x)=Q(x)e^{-x}$, which satisfies
\begin{equation}\label{may2830}
\rho  Z= -c \frac{\partial Z}{\partial x}+ 
\kappa\frac{\partial^2 Z}{\partial x^2}+2\kappa\pdr{Z}{x}+\kappa Z+H(R).
\end{equation}
Note that for $A_1\le Z(x)\le A_2$ and  for
$x>x_0^+$ we have $R(x)<1$, so that 
\[
H(R)=1+R^2,
\]
and (\ref{may2830}) becomes
\begin{equation}
(\rho-\kappa)  Z= -(c-2\kappa) \frac{\partial Z}{\partial x}+ 
\kappa\frac{\partial^2 Z}{\partial x^2}+1+R^2,\hbox{ for $x>x_0^+$.}
\end{equation}
Let us assume that~$y_n\to+\infty$ such that $Z(y_n)\to\zeta$ as $n\to+\infty$.
Since $Z(x)$ is uniformly bounded and positive, (\ref{may2830}) implies
that $\|Z\|_{C^{2,\alpha}}\le C$, hence
the functions $Z_n(x)=Z(x+y_n)$ converge, after extracting a subsequence to a 
function $\bar Z$. As $R(x)\to 0$ as $x\to+\infty$, the function $\bar Z(x)$ is a
bounded solution to
\begin{equation}\label{may2832}
(\rho-\kappa)  \bar Z= -(c-2\kappa) \frac{\partial\bar Z}{\partial x}+ 
\kappa\frac{\partial^2 \bar Z}{\partial x^2}+1,\hbox{ for $x\in\Rm$}. 
\end{equation}
It follows that $\bar Z(x)\equiv \kappa/(\rho-\kappa)$, and, in particular, 
$\zeta=1/(\rho-\kappa)$, finishing the proof.   
%
%
\end{proof}
This also completes the proof of Theorem~\ref{2.1}, except for the strict inequality in the upper bound $c<2\sqrt{\kappa\alpha}$. 

\subsubsection*{The proof of Proposition~\ref{prop-speed-2.2}}

We prove the matching lower and upper bounds on $c$. First, exactly as in the proof of Lemma~\ref{cboundtau}, using an exponential 
super-solution and the normalization at $x=0$, we can show that
for any~$\eps>0$ there exists $a_0>0$ such that
\begin{equation}\label{jun2802}
-\eps<c^a< 2\sqrt{\kappa\int_{-a}^a s_a^*(y)(-F_y^a(y))dy}+\eps \mbox{ for all } a>a_0.
\end{equation}
Passing to the limit $a\to+\infty$, we get an upper bound
\begin{equation}\label{jun2808}
c\le 2\sqrt{\kappa\int_{-\infty}^\infty s^*(y)(-F_y(y))dy}.
\end{equation}
For the lower bound, let $F(x)$ be the solution to the traveling wave equation
\begin{equation}\label{jun2804}
-cF_x-\kappa F_{xx}=\alpha F(x)\int_{-\infty}^{x}s^*(y)(-F_y(y))dy,
\end{equation}
with $F(-\infty)=1$ and $F(+\infty)=0$ that we have just constructed, and set $F^n(x)=F(x+n)/F(n)$. The functions $F_n$ satisfy
\begin{equation}\label{jun2806}
-cF^n_x-\kappa F_{xx}^n=\alpha \gamma_n(x)F^n,~~\gamma_n(x)=\int_{-\infty}^{x+n}s^*(y)(-F_y(y))dy,
\end{equation}
with $F_n(0)=1$. The standard elliptic regularity estimates and the Harnack inequality imply 
that after extracting a subsequence, the functions $F_n(x)$ converge locally uniformly to a limit $G(x)$ that satisfies
\begin{equation}\label{jun2810}
-cG_x-\kappa G_{xx}^n=\alpha \gamma G,~~\gamma=\int_{-\infty}^{\infty}s^*(y)(-F_y(y))dy>0,
\end{equation}
and $G(0)=1$. In addition, the function $G(x)$ is positive and monotonically decreasing. As a consequence, since $c>0$, we must have
\begin{equation}\label{jun2812}
c\ge 2\sqrt{\kappa\gamma},
\end{equation}
and the proof of Proposition~\ref{prop-speed-2.2} is complete.~$\Box$

\section{Numerical results}\label{Numerics}

In this section, we describe numerical results obtained via an iterative finite 
differences scheme for  
the traveling wave system on a finite interval $[-a,a]$, using the following 
algorithm:
\begin{itemize}
\item Start with an initial guess for $F(x)$ and $Q(x)$,
compute  $R(x)$ from (\ref{may1510}), and
$H(x)$ using~(\ref{Hdef2}), and set 
$s^*(x)=\min(1,R(x))$. 
A suitable initial guess for $F(x)$ is the traveling wave
solution of the Fisher-KPP equation  on~$[-a,a]$. 
We can take $Q(x)=e^x$ as an initialization. 
\item Given $H(x)$, we  solve (\ref{Qeq})  for $Q(x)$ on $[-a,a]$ with the boundary conditions (\ref{BCQ}). 
\item Given $s^*(x)$ we solve (\ref{Feq}) for $F(x)$ and $c$, with the boundary conditions (\ref{BCF}) and normalization
(\ref{may1512}), using an iterative finite different scheme. 
\end{itemize}

Recall that both $Q(x)$ and $R(x)$ grow exponentially, on the right and on the left, respectively. 
Accordingly, we rescale them, so that all functions involved are bounded. As the equation for $F$ is non-linear, we use another iterative finite difference  scheme to solve it, with a
modified boundary condition that uses the solution of the linearized problem. 

\subsection{A rescaling for  $Q(x)$ and $R(x)$}\label{rescaling}

As $Q(x)$ approaches a positive constant on the left,
simply rescaling 
it by  $e^{-x}$ to remove the exponential growth on the right would lead to
an exponential growth on the left for the rescaled function.
Instead, define  
\begin{equation}\label{defg}
g(x)=
\begin{cases}
1+2\tan^{-1}(1)-2\tan^{-1}(x+1)~\mbox{ for } x\leq 0,\\
e^{-x} ~ \mbox{ for } x\geq 0.
\end{cases}
\end{equation}
The function $g(x)$ is continuous, with continuous first three derivatives, 
converges to a constant on the left and 
decays exponentially on the right. 
Then $\tilde{Q}(x)=g(x)Q(x)$ satisfies
\[
\phi_1 \tilde{Q}+\phi_2\tilde{Q}_x-\kappa \tilde{Q}_{xx}=ge^xH(R(x))],
\]
with
%
%
\begin{align*}
\phi_1&=\rho-c-c\frac{g_1}{g}-\frac{\kappa(2g^2_1-gg_2)}{g^2}=\rho-\kappa~\mbox{ for } x>0,\\
\phi_2&=c+2\frac{\kappa g_1}{g}=c-2\kappa~\mbox{ for } x>0.
\end{align*}
Here, $g_1$ and $g_2$ are, respectively, the first and the second derivative of $g$. 
Note that for $x\geq 0$ the  functions $\phi_1$ and $\phi_2$ are constants and for $x<0$ they 
are bounded. 
%

The boundary conditions $Q'(-a)=0$ and $Q'(a)=Q(a)$ become
\[\tilde{Q}'(-a)=\frac{g_1(-a)}{g(-a)}\tilde{Q}(-a)~\mbox{ and }\tilde{Q}'(a)=0.\]
We solve numerically for $\tilde{Q}$ using a finite difference scheme. 


We use a similar strategy to rescale $R(x)$. Note that for $x<0$ we have that $R(x)\sim e^{-x}$.  The function
$\tilde{R}(x)=e^xR(x)$ 
satisfies
\begin{equation}\label{feb1102}
\tilde{R}'(x)=-\frac{\alpha}{2}Q'(x)F(x)
\end{equation}
with the boundary condition $\tilde{R}(a)=0$.
We approximate~$Q'(x)$ in (\ref{feb1102})~as   
 \[
 Q'(x_i)=\frac{1}{g}\tilde{Q}_x-\frac{g_1}{g^2}\tilde{Q}\approx \frac{1}{g}\frac{\tilde {Q}_{i+1}-\tilde{Q}_{i-1}}{2h}-\frac{g_1}{g^2}\tilde{Q}_i.
 \] 
%

\subsection{Numerical Solution for $F(x)$ on the interval $[-a,a]$ }

Let us first briefly explain an iterative finite difference scheme and relaxed boundary conditions  for the Fisher-KPP equation
\begin{equation}\label{jun2602}
-cF_x=\kappa F_{xx}+\alpha F(1-F),
\end{equation}
on a finite interval $[-a,a]$, with $F(-a)=1$, $F(a)=0$ and $F(0)=1/2$.  
This corresponds to~$s^*(x)\equiv 1$. 
Near $x=a$, the solution to the Fisher-KPP equation is well approximated by the linearized equation
\begin{equation}\label{jun2604}
-cF_x-\kappa F_{xx}=\alpha F.
\end{equation}
A solution to (\ref{jun2604}) with the initial condition $F(0)=1/2$ is
\[
F(x)=\frac{1}{2}\exp\{-\beta x\},
\]
where $\beta $ is given by
\begin{equation}\label{jun2702}
\kappa \beta^2-c\beta +\alpha = 0,~~
~~\beta(c)=
\left\{
\begin{matrix}
\dfrac{c}{2\kappa} ~ \mbox{ if }0<c<2,\cr
\dfrac{c-\sqrt{c^2-4\kappa\alpha}}{2\kappa}~\mbox{ if } c\geq 2.\cr
\end{matrix}\right.
\end{equation}
We use 
\begin{equation}\label{jun2702bis}
F(a)=\frac{1}{2}\exp\{-\beta(c) a\},
\end{equation}
as a new boundary condition for the Fisher-KPP equation on $[-a,a]$, instead of $F(a)=0$. 
It will change at each step of the iterative algorithm to reflect the change in $c$.
{The modified boundary condition speeds up the rate of convergence of the solution
to the iterative scheme described below.} 

The iterative algorithm solves 
\begin{equation}\label{kpp}
-cF_x^{k+1}-\kappa F_{xx}^{k+1}=\alpha F^k(1-F^k),
\end{equation}
on the interval $[-a,a]$, with the boundary condition $F^{k+1}(-a)=1$, 
$F^{k+1}(a)= (1/2)\exp\{-\beta(c_k) a\}$ with the initialization 
$F_0(x)=(a-x)/(2a)$. 
The first and the second derivatives are approximated using the central differences. The speed $c$ is updated after each iteration, to enforce $F^k(0)={1}/{2}$ at each iteration step. 
 
%
%
%
%
%
%
%
%

\subsubsection*{Solution for $F(x)$ for general $s^*(x)$}

We solve numerically the equation:
\begin{equation}\label{kpp-bis}
-cF_x-\kappa F_{xx}=\alpha F\int_{-a}^xs^*(y)(-F_y(y))dy
\end{equation}
on the interval $[-a,a]$ with the boundary condition  
as in the scheme for   the Fisher-KPP equation. 
Take a partition with step $h$ and let $n=2a/h$,
$x_i=-a+ih$ and  $F^k_i=F^k(x_i)$ be the value of the~$k$-th approximation of the solution at $x_i$. At each step of the iterative scheme the speed $c_k$ is updated, so that  $F^k(0)=1/2$.
We take $F^0$ to be the solution of the Fisher-KPP equation
\[
-c_0F_x^0-\kappa F_{xx}^0=\alpha F^0(1-F^0)
\]
on the interval $[-a,a]$, with $c_0$ chosen, so that $F^0(0)=1/2$. 
Using the central differences, we approximate the left side of (\ref{kpp-bis}) as
\[
-c_k\frac{F^k_{i+1}-F^k_{i-1}}{2h}-\kappa\frac{F^k_{i+1}-2F^k_i+F^k_{i-1}}{h^2}.
\]
To reduce the error in the approximation of the right side of (\ref{kpp-bis}), we  integrate by parts  
\begin{align*}
&\int_{-a}^{x_i}s^*(y)(-F_y(y))dy = \int_{-a}^{x_0}s^*(y)(-F_y(y))dy+\int_{x_0}^{x_i}s^*(y)(-F_y(y))dy\\
&=1-F(x_0)+\int_{x_0}^{x}s^*(y)(-F_y(y))dy
=1-F(x_0)+F(x_0) - F(x_i)s^*(x_i) - \int_{x_0}^{x_i}s_y^*(y)(-F(y))dy\\
&=1-F(x_i)R(x_i) +\int_{x_0}^{x_i}R_y(y)F(y)dy\approx 1-F^k_iR_i +\sum_m^{i-1}h\frac{R_y(x_j)F^k_j+R_y(x_{j+1})F^k_{j+1}}{2}.
\end{align*}
In the last computation we take $x_0=\max\{x|s^*(x)=1\}$ and $m$ such that $x_0=-a+mh$. 
We also use that $s^*(x)=\min\{1,R(x)\}$, so for $x>x_0$ we have that $s^*_x(x)=R_x(x)$. 
Thus, the discretized version of (\ref{kpp-bis}) is
\begin{equation}\label{jun2806bis}
-c_k\frac{F^k_{i+1}-F^k_{i-1}}{2h}-\kappa\frac{F^k_{i+1}-2F^k_i+F^k_{i-1}}{h^2}=\alpha F^{k-1}_i\Big(1-F^k_iR_i +\sum_m^{i-1}h\frac{R_y(x_j)F^k_j+R_y(x_{j+1})F^k_{j+1}}{2}\Big).
\end{equation}

%

\subsection{Discussion of the numerical results}\label{sec:num-dis}

We now discuss some conclusions one can draw from the numerical
simulations.  In particular, we compare the traveling wave profile of the 
knowledge distribution function to the traveling wave profile of the Fisher-KPP 
equation, and study numerically the dependence of the traveling wave solution 
of the mean field system on the parameters in the model.

\subsubsection*{Convergence as $a$ increases}

We first illustrate the convergence of the numerical scheme for 
the mean field system on a finite interval $[-a,a]$ as $a$ increases.  
We fix the system parameters $\kappa=1$, $\alpha=2$ and  $\rho=10$,
and the discretization step $h=0.02$,  and 
consider $a=\{15,20,25,30,35,40\}$.  
As expected, we 
observe the solutions for $F$ and $Q$ 
converge pointwise as $a$ grows. As the plots illustrating the convergence of $F$
and $Q$ do not seem very informative or surprising, we present below the 
convergence in $a$ of various objects associated to these functions. 
Convergence of the transition point 
$x_0$ as $a$ increases is illustrated in Figure~\ref{fig:x0-a}.   

\begin{figure}[H]
    \centering
    \includegraphics[width=0.5\textwidth]{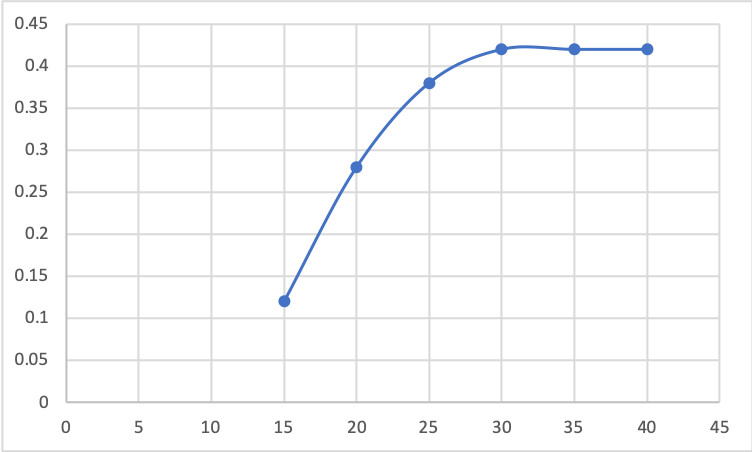}
    \caption{The transition point $x_0$ for $a\in\{15,20,25,30,35,40\}$}\label{fig:x0-a}
\end{figure}

Recall that, according to Theorem~\ref{2.1},
the function $Q(x)$ approaches a positive limit $Q_->0$ on the left, and
a multiple~$Q_+e^x$ of an exponential 
on the right, so that the rescaled function~$\tilde Q$ approaches
$Q_-$ and $Q_+$ on the left and the right, respectively. 
This is illustrated in Figure \ref{fig:q40}.
\begin{figure}[H]
    \centering
    \includegraphics[width=0.5\textwidth]{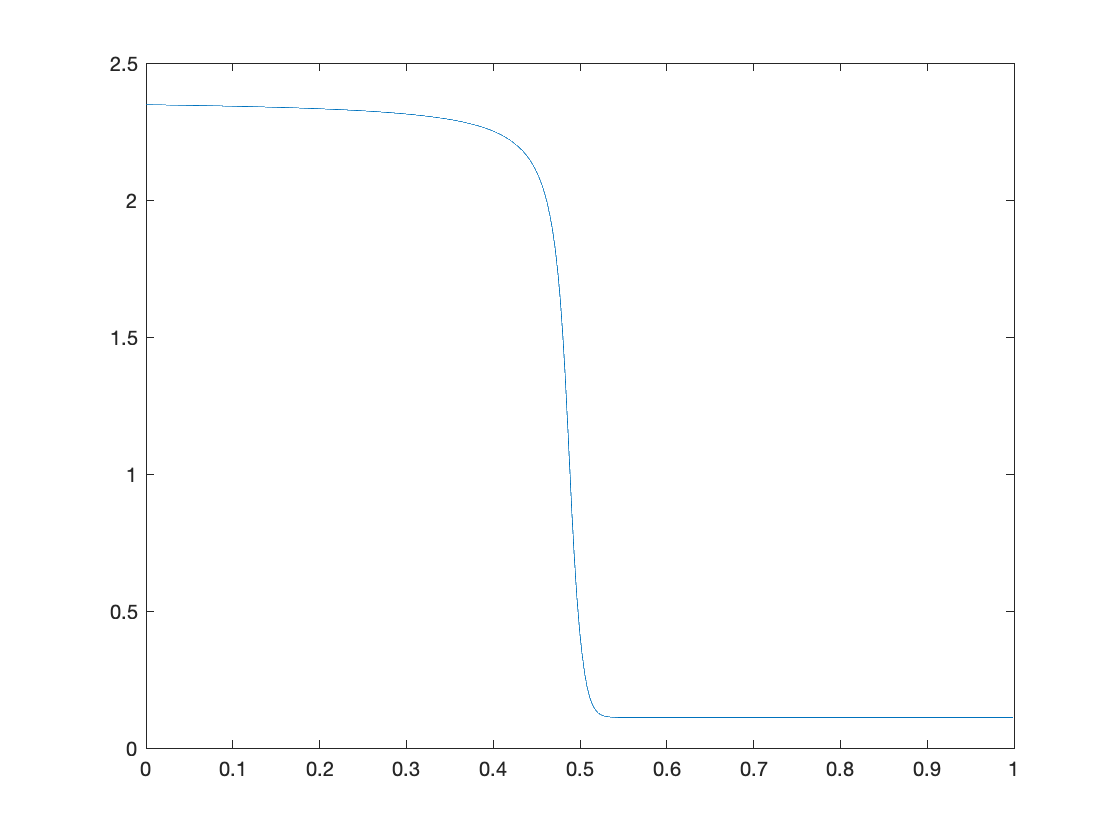}
    \caption{Rescaled $Q$ on $[-40,40]$, with $a=40$.}\label{fig:q40}
\end{figure}
Convergence of $Q_-$ and $Q_+$ as $a$ 
increases is illustrated in Figures \ref{fig:ql-a}, \ref{fig:qr-a}.
Note that $Q_+$ is very close to the theoretical value $1/(\rho-\kappa)$, as in the third line 
of (\ref{intrEq4}) in Theorem~\ref{2.1}.

\begin{figure}[H]
    \centering
    \includegraphics[width=0.5\textwidth]{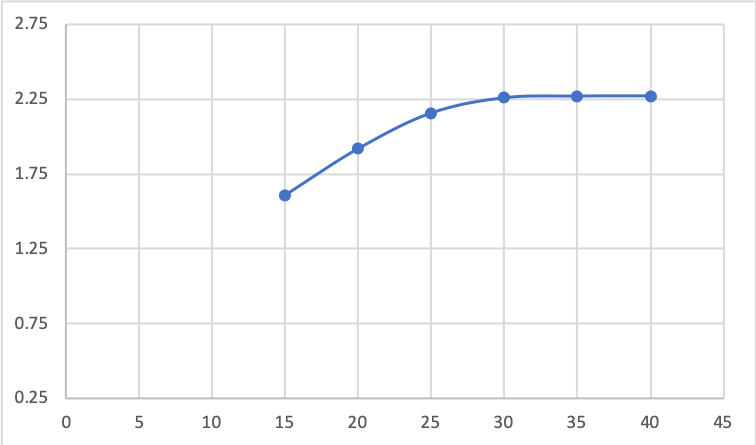}
    \caption{The left limit $Q_-$ for $a\in\{15,20,25,30,35,40\}$.}\label{fig:ql-a}
\end{figure}

\begin{figure}[H]
    \centering
    \includegraphics[width=0.5\textwidth]{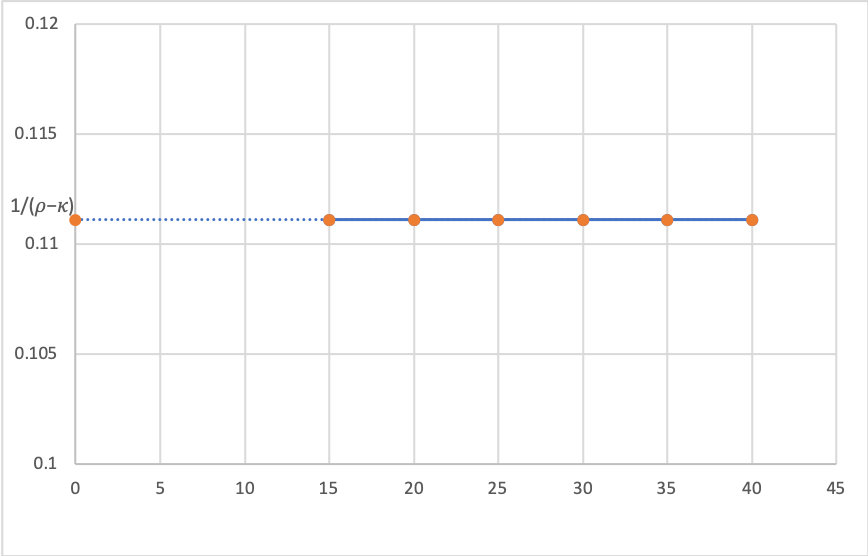}
    \caption{The right limit $Q_+$ for $a\in\{15,20,25,30,35,40\}$.}\label{fig:qr-a}
\end{figure}


The search function for $a=40$ is plotted in Figure \ref{fig:sa-35} below. 
\begin{figure}[H]    
\centering
    \includegraphics[width=0.5\textwidth]{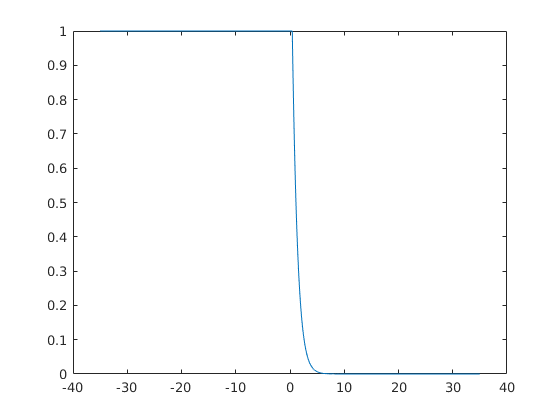}
    \caption{$s^*(x)$ on $[-40,40]$}\label{fig:sa-35}
\end{figure}



\subsubsection*{Comparison to a Fisher-KPP wave}

Let us now compare the traveling wave solution $F$ for the mean field system to a 
Fisher-KPP traveling wave with the same values of $\alpha=2$ and $\kappa=1$. 
Note that the Fisher-KPP speed with these parameters 
is $c_{FKPP}=2\sqrt{2}$. The mean field system speed is smaller.  
The speed of the mean field system for different values of $a$ 
is represented in the following plot: 
\begin{figure}[H]
    \centering
    \includegraphics[width=0.5\textwidth]{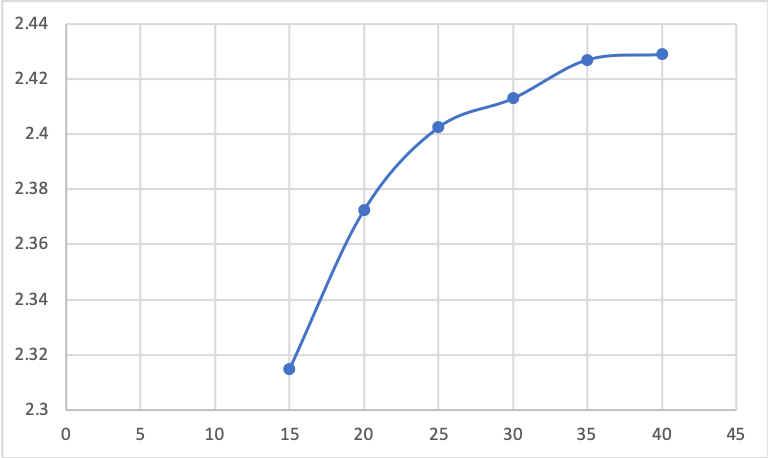}
    \caption{The traveling wave speed for $a\in\{15,20,25,30,35,40\}$.}
\end{figure}

We plot the solution for  $F$  and  the solution of the Fisher-KPP equation with 
$\alpha=2$ Figure~\ref{fig:FvsKPP}. As expected,~$F$ 
is above the Fisher-KPP solution on the left and below on the right, although they are very close. We will later see   that the discrepancy becomes larger as $\rho$ increases.
\begin{figure}[H]
    \centering
    \includegraphics[width=0.5\textwidth]{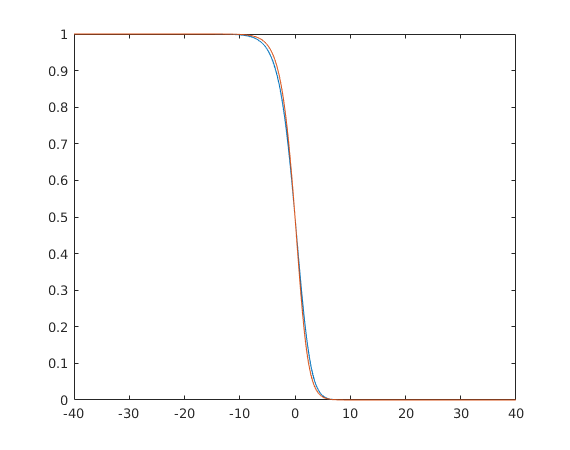}
    \caption{The Fisher-KPP solution vs $F$ on $[-40,40]$}\label{fig:FvsKPP}
\end{figure}
The difference between the two profiles is better seen when the level slope of the function  $F$  is plotted  vs the level slope of the solution of the Fisher-KPP equation with $\alpha=2$ in the Figure~\ref{fig:FvsKPP-slope}. We observe that $F$ is steeper for all values of $F(x)$. 
\begin{figure}[H]
    \centering
    \includegraphics[width=0.5\textwidth]{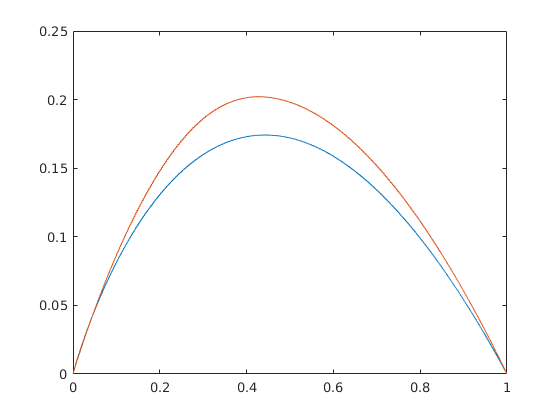}
    \caption{The level slope comparison - $KPP$ vs $F$ on $[-40,40]$}\label{fig:FvsKPP-slope}
\end{figure}

In the rest of this section we study numerically how the three parameters in the model affect 
the wave profiles and the speed of the wave.

\subsubsection*{Dependence of the solution on $\alpha$}

Recall that the probability that an agent, who spends time $s$ searching, meets another 
agent over a time interval $\Delta t$ is $\alpha \sqrt{s}\Delta t$. 
Therefore, the parameter $\alpha$ corresponds to the effectiveness of the search 
-- the larger $\alpha$, the easier it is for agents to meet other agents 
and increase their productivity via learning. We now
fix the values of $\kappa=1$ and $\rho=10$, the discretization
$h=0.02$ and $a=40$ and illustrate numerically the dependence of 
the solution of the mean field system on~$\alpha$.

As seen from Figure~\ref{fig:c-alpha},
the speed of the traveling wave increases as $\alpha$ increases, but not as fast 
as the minimal speed for the Fisher-KPP equation given by $c_{FKPP}=2\sqrt{\kappa\alpha}$. 
We also observe that as $\alpha$ 
approaches from above the critical value $\alpha_c=1$, below which traveling 
weaves do not exist,   the speed tends to the Fisher-KPP speed, corresponding 
to $s^*\equiv 1$. 
This is as expected: as the search effectiveness decreases and 
approaches $\alpha_c$, for the economy to grow along a balanced path, the agents
need to spend more and more time searching, so that the transition point
$x_0$ would tend to $+\infty$.  As a reference, we also display in Figure~\ref{fig:c-alpha}
the lower bound
$c_{low}=2\kappa$ that holds for the traveling wave speed for any~$\alpha>0$. 
\begin{figure}[H]
    \centering
    \includegraphics[width=0.5\textwidth]{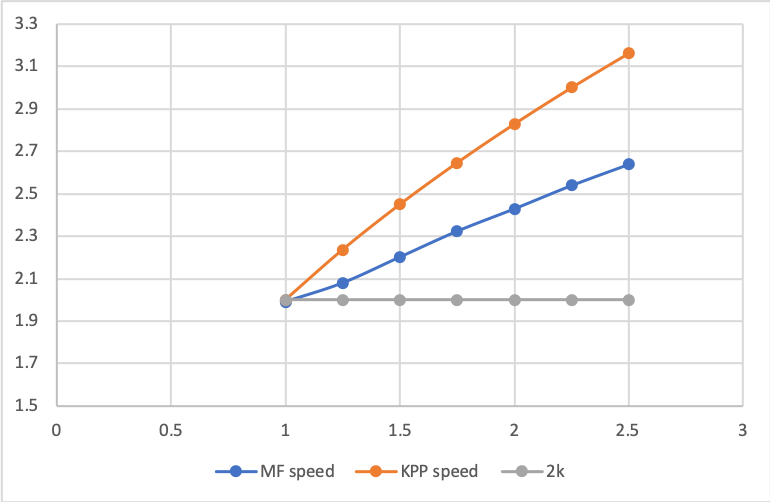}
    \caption{Dependence of the speed on $\alpha$}\label{fig:c-alpha}
\end{figure}

The dependence of $x_0$ on $\alpha$ is also illustrated in Figure~\ref{fig:x0-alpha}.

\begin{figure}[H]
    \centering
    \includegraphics[width=0.5\textwidth]{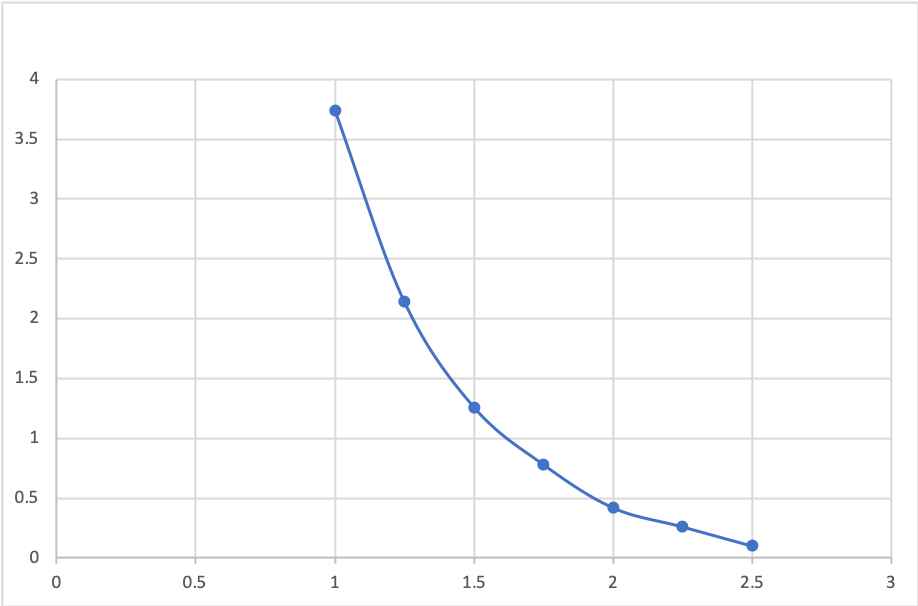}
    \caption{Dependence of $x_0$ on $\alpha$}\label{fig:x0-alpha}
\end{figure}

\subsubsection*{Dependence of the solution on $\rho$}

Next, we consider numerically the dependence of the solution on the
interest rate $\rho$. 
We fix the rest of the parameters  in the simulation as $\kappa=1$,
$\alpha=2$, $h=0.02$ and $a=40$. We observe that as~$\rho$ increases both the speed 
and $x_0$ decrease. This corresponds to a slower growth of the economy 
and to slower learning: agents tend to produce and not search as production 
is more 
profitable than the heavily discounted benefit of learning. 
 In particular, we see that 
when $\rho$ increases, the wave speed approaches its lower bound 
$c_{low}=2\kappa$.
On the other hand, we also see in Figure~\ref{fig:c-rho} that as $\rho$ 
approaches from above
the critical value
$\rho_c=\kappa$, the speed approaches the Fisher-KPP speed, corresponding to  
$x_0\rightarrow+\infty$. Numerically, we see in Figure~\ref{fig:x0-rho}
see that in this case
$x_0$ also increases.     

\begin{figure}[H]
    \centering
    \includegraphics[width=0.5\textwidth]{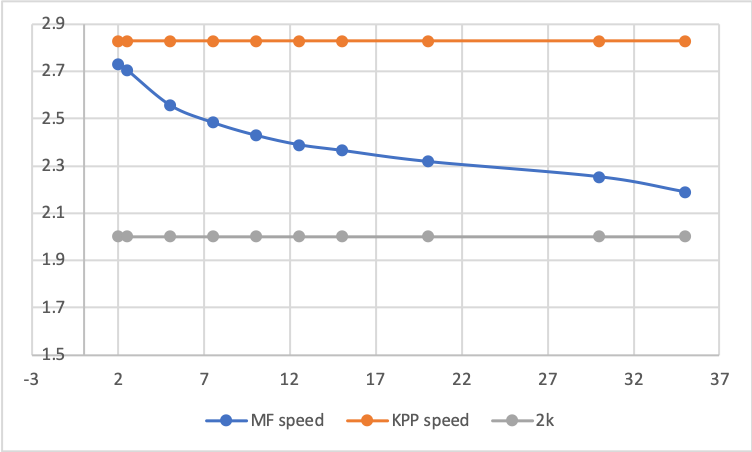}
    \caption{Dependence of the speed on $\rho$}\label{fig:c-rho}
\end{figure}

\begin{figure}[H]
    \centering
    \includegraphics[width=0.5\textwidth]{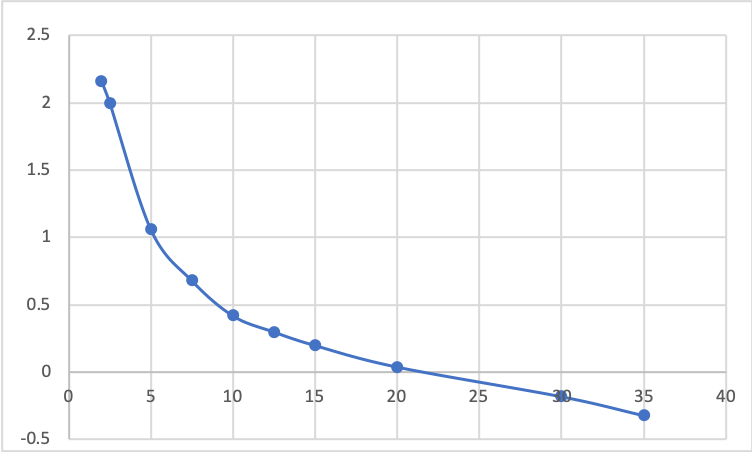}
    \caption{Dependence of $x_0$ on $\rho$}\label{fig:x0-rho}
\end{figure}

In Figures~\ref{fig:FKPP-rho1.3} and~\ref{fig:FKPP-rho35} we plot the numerical solution for $F$ against the numerical solution for the Fisher-KPP equation for $a=40$ and for different values of $\rho$. We only consider~$x\in[-5,5]$ as both solutions approach very fast
the values $1$ on the left and $0$ on the right
outside of this interval. We observe that the numerical solution of $F$ gets closer to the numerical solution of the Fisher-KPP equation as $\rho$ decreases.

\begin{figure}[H]
    \centering
    \includegraphics[width=0.5\textwidth]{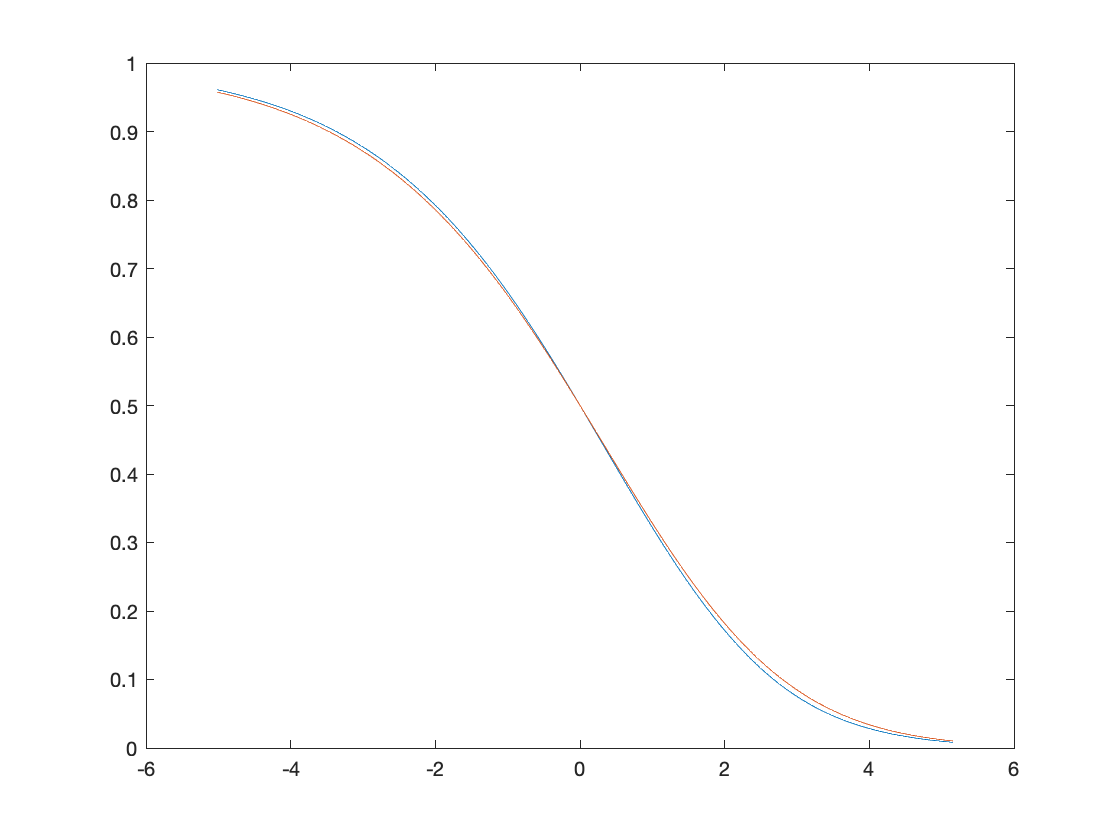}
    \caption{The traveling wave profile F vs the Fisher-KPP solution for $\rho=2.2$.}\label{fig:FKPP-rho1.3}
\end{figure}

\begin{figure}[H]
    \centering
    \includegraphics[width=0.5\textwidth]{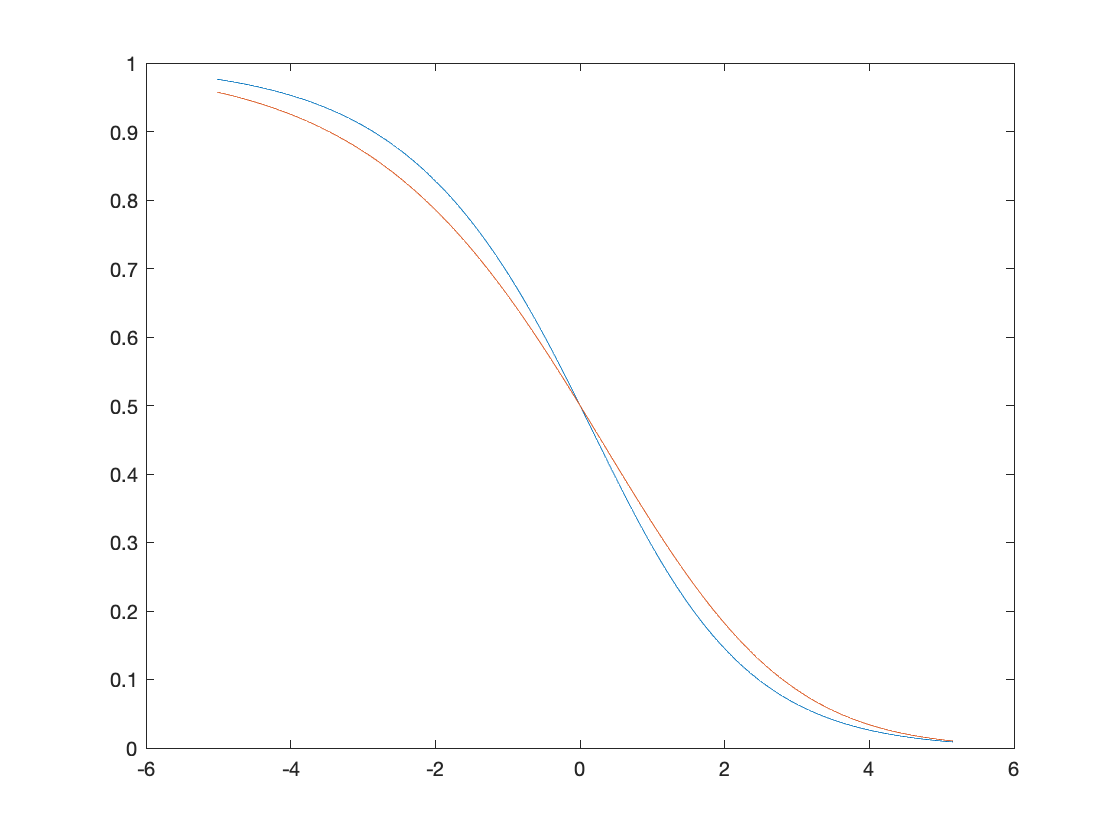}
    \caption{The traveling wave profile F vs the Fisher-KPP solution for for $\rho=35$.}\label{fig:FKPP-rho35}
\end{figure}

\subsection*{Dependence of the solution on $\kappa$}

Finally, we look at the dependence of the solution on the diffusivity $\kappa$.  We fix the 
parameters  in the simulation as  $\alpha=2$, $\rho=10$, $h=0.02$, 
and $a=40$.  We observe that as $\kappa$ increases the speed increases 
and the transition point $x_0$ moves to the right. This is intuitive from the
economics point of view as large $\kappa$ induces a fat tail of the distribution of knowledge, so agents will have higher incentive to search, as meeting an agent with a high'
productivity and thus increasing your own productivity will be easier.  

\begin{figure}[H]
    \centering
    \includegraphics[width=0.5\textwidth]{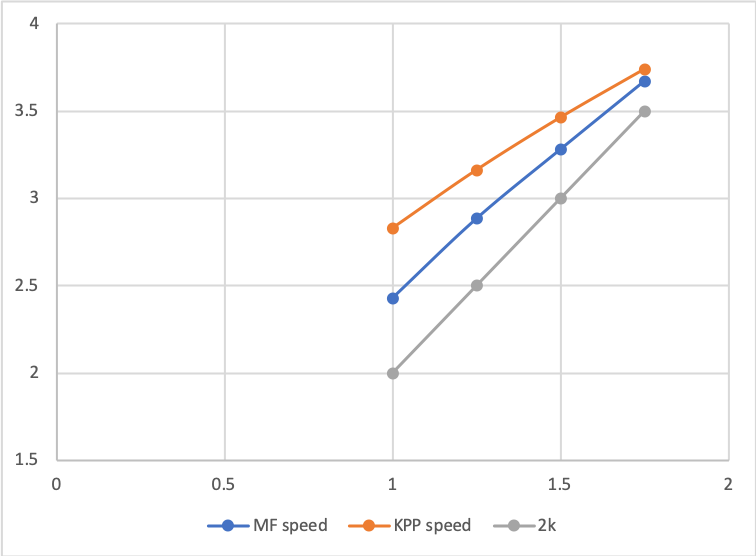}
    \caption{Dependence of the speed on $\kappa$.}
\end{figure}

\begin{figure}[H]
    \centering
    \includegraphics[width=0.5\textwidth]{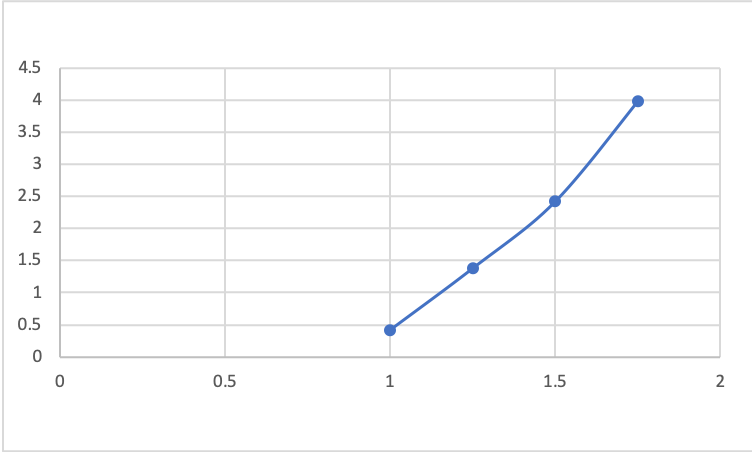}
    \caption{Dependence of $x_0$ on $\kappa$.}
\end{figure}

\end{document}